\documentclass{amsart}

\usepackage[nocompress]{cite}

\usepackage[all]{xy}
\usepackage{amssymb,mathrsfs, amscd, graphicx, array, multirow,
enumerate, amsfonts, euscript}

\allowdisplaybreaks

\usepackage{amscd}

\usepackage{stmaryrd}
\newcommand{\dbr}[1]{\left\llbracket #1 \right\rrbracket}
\usepackage{hyperref}

\makeatletter
\def\@tocline#1#2#3#4#5#6#7{\relax
  \ifnum #1>\c@tocdepth 
  \else
    \par \addpenalty\@secpenalty\addvspace{#2}%
    \begingroup \hyphenpenalty\@M
    \@ifempty{#4}{%
      \@tempdima\csname r@tocindent\number#1\endcsname\relax
    }{%
      \@tempdima#4\relax
    }%
    \parindent\z@ \leftskip#3\relax \advance\leftskip\@tempdima\relax
    \rightskip\@pnumwidth plus4em \parfillskip-\@pnumwidth
    #5\leavevmode\hskip-\@tempdima
      \ifcase #1
       \or\or \hskip 1em \or \hskip 2em \else \hskip 3em \fi%
      #6\nobreak\relax
    \dotfill\hbox to\@pnumwidth{\@tocpagenum{#7}}\par
    \nobreak
    \endgroup
  \fi}
\makeatother

\usepackage{letltxmacro}
\LetLtxMacro{\oldsqrt}{\sqrt}
\renewcommand{\sqrt}[2][]{\,\oldsqrt[#1]{#2}\,}

\def\bmu{\boldsymbol \mu}

\newcommand{\abs}[1]{\lvert #1 \rvert}
\newcommand{\zmod}[1]{\mathbb{Z}/ #1 \mathbb{Z}}

\newcommand{\dangle}[1]{\left\langle #1 \right\rangle}

\newcommand{\Lsymb}[2]{\genfrac{(}{)}{}{}{#1}{#2}}    

\newcommand{\dia}{\diamond}

\def\Sq{{\mathrm Sq}}
\def\odd{{\mathrm odd}}

\DeclareMathSymbol{\twoheadrightarrow} {\mathrel}{AMSa}{"10}

\DeclareMathOperator{\ord}{ord}

\DeclareMathOperator{\Emb}{Emb}

\DeclareMathOperator{\Pic}{Pic}

\DeclareMathOperator{\End}{End}
\DeclareMathOperator{\Hom}{Hom}

\DeclareMathOperator{\Mat}{Mat}

\DeclareMathOperator{\Tr}{Tr}
\DeclareMathOperator{\Nm}{N}  

\DeclareMathOperator{\Sp}{Sp}

\newcommand{\scrB}{\mathscr{B}}

\newcommand{\scrD}{\mathscr{D}}

\newcommand{\scrK}{\mathscr{K}}

 \newcommand{\ff}{\mathbb{F}}

\newcommand{\calO}{\mathcal{O}}

\newcommand{\wh}{\widehat}
\DeclareMathOperator{\Cl}{Cl}

\DeclareMathOperator{\Mass}{Mass}
\DeclareMathOperator{\Nr}{Nr}

\newcommand{\wt}{\widetilde}

\newcommand{\grp}{\mathfrak{p}}

\newcommand{\grf}{\mathfrak{f}}

\def\makeop#1{\expandafter\def\csname#1\endcsname
  {\mathop{\rm #1}\nolimits}\ignorespaces}
\makeop{Hom}   \makeop{End}   \makeop{Aut}   \makeop{Isom}  \makeop{Pic} 
\makeop{Gal}   \makeop{ord}   \makeop{Char}  \makeop{Div}   \makeop{Lie} 
\makeop{PGL}   \makeop{Corr}  \makeop{PSL}   \makeop{sgn}   \makeop{Spf}
\makeop{Spec}  \makeop{Tr}    \makeop{Nr}    \makeop{Fr}    \makeop{disc}
\makeop{Proj}  \makeop{supp}  \makeop{ker}   \makeop{im}    \makeop{dom}
\makeop{coker} \makeop{Stab}  \makeop{SO}    \makeop{SL}    \makeop{SL}
\makeop{Cl}    \makeop{cond}  \makeop{Br}    \makeop{inv}   \makeop{rank}
\makeop{id}    \makeop{Fil}   \makeop{Frac}  \makeop{GL}    \makeop{SU}
\makeop{Nrd}   \makeop{Sp}    \makeop{Tr}    \makeop{Trd}   \makeop{diag}
\makeop{Res}   \makeop{ind}   \makeop{depth} \makeop{Tr}    \makeop{st}
\makeop{Ad}    \makeop{Int}   \makeop{tr}    \makeop{Sym}   \makeop{can}
\makeop{length}\makeop{SO}    \makeop{torsion} \makeop{GSp} \makeop{Ker}
\makeop{Adm}   \makeop{Mat}

\def\makebb#1{\expandafter\def
  \csname bb#1\endcsname{{\mathbb{#1}}}\ignorespaces}
\def\makebf#1{\expandafter\def\csname bf#1\endcsname{{\bf
      #1}}\ignorespaces} 
\def\makegr#1{\expandafter\def
  \csname gr#1\endcsname{{\mathfrak{#1}}}\ignorespaces}
\def\makeescr#1{\expandafter\def
  \csname escr#1\endcsname{{\EuScript{#1}}}\ignorespaces}
\def\makecal#1{\expandafter\def\csname cal#1\endcsname{{\mathcal
      #1}}\ignorespaces} 

\def\doLetters#1{#1A #1B #1C #1D #1E #1F #1G #1H #1I #1J #1K #1L #1M
                 #1N #1O #1P #1Q #1R #1S #1T #1U #1V #1W #1X #1Y #1Z}
\def\doletters#1{#1a #1b #1c #1d #1e #1f #1g #1h #1i #1j #1k #1l #1m
                 #1n #1o #1p #1q #1r #1s #1t #1u #1v #1w #1x #1y #1z}
\doLetters\makebb   \doLetters\makecal  \doLetters\makebf

\doLetters\makeescr 

\doletters\makebf   \doLetters\makegr   \doletters\makegr

\def\ol{\overline}
\def\wt{\widetilde}

\def\wh{\widehat}

\def\Fpbar{\overline{\bbF}_p}

\def\Fq{{\bbF}_q}

\newcommand{\Z}{\mathbb Z}
\newcommand{\Q}{\mathbb Q}
\newcommand{\R}{\mathbb R}
\newcommand{\C}{\mathbb C}
\newcommand{\F}{\mathbb F}

\newcounter{thmcounter} 
\numberwithin{thmcounter}{section}
\newtheorem{thm}[thmcounter]{Theorem}
\newtheorem{lem}[thmcounter]{Lemma}
\newtheorem{lemma}[thmcounter]{Lemma}
\newtheorem{cor}[thmcounter]{Corollary}
\newtheorem{prop}[thmcounter]{Proposition}
\theoremstyle{definition}

\newtheorem{ex}[thmcounter]{Example}
\newtheorem{notn}[thmcounter]{Notation}

\newtheorem{rem}[thmcounter]{Remark}
\newtheorem{remark}[thmcounter]{Remark}

\newtheorem{sect}[thmcounter]{}

\numberwithin{equation}{section}

\newtheoremstyle{notitle}  
  {}
  {}
  {\itshape}
  {}
  {}
  {\ }
  {.5em}
  {}
\theoremstyle{notitle}

\title[Superspecial abelian surfaces]{On superspecial abelian surfaces
  and type numbers of totally definite quaternion algebras}
\author{Jiangwei Xue and Chia-Fu Yu}

\address{(Xue) Collaborative Innovation Centre of Mathematics, 
School of Mathematics and Statistics, Wuhan University, Luojiashan,
 Wuhan, Hubei, 430072, P.R. China.}
\email{xue\_j@whu.edu.cn}

\address{(Yu) Institute of Mathematics,
  Academia Sinica and NCTS, Astronomy-Mathematics
  Building, No. 1, Sec. 4, Roosevelt Road, Taipei 10617, TAIWAN.}
\email{chiafu@math.sinica.edu.tw} 

\begin{document}

\date{\today} 

\begin{abstract}
  In this paper we determine the number of endomorphism rings of
  superspecial abelian surfaces over a field $\Fq$ of odd degree over
  $\F_p$ in the isogeny class corresponding to the Weil $q$-number
  $\pm\sqrt{q}$.  This extends earlier works of T.-C. Yang and the
  present authors on the isomorphism classes of these abelian
  surfaces, and also generalizes the classical formula of Deuring for
  the number of endomorphism rings of supersingular elliptic curves.
  Our method is to explore the relationship between the type and class
  numbers of the quaternion orders concerned. We study the Picard
  group action of the center of an arbitrary $\Z$-order in a totally
  definite quaternion algebra on the ideal class set of said order,
  and derive an orbit number formula for this action. This allows us
  to prove an integrality assertion of Vign\'eras
  [Enseign. Math. (2), 1975] as follows.  Let $F$ be a totally real
  field of even degree over $\Q$, and $D$ be the (unique up to
  isomorphism) totally definite quaternion $F$-algebra unramified at
  all finite places of $F$. Then the quotient $h(D)/h(F)$ of the class
  numbers is an integer.

 
 


\end{abstract}

\maketitle


\def\Bp{{D_{p,\infty}}}

\section{Introduction}
\label{sec:01}

Throughout this paper $p$ denotes a prime number. 
Let $\Bp$ denote the unique definite quaternion $\Q$-algebra up to
isomorphism ramified exactly at $p$ and $\infty$.
The classical result of Deuring establishes a bijection
between the set $\Lambda_p$ of 
isomorphism classes of supersingular
elliptic curves over $\Fpbar$ and 
the set $\Cl(\Bp)$ of 
ideal classes of a maximal order in $\Bp$. 
The class number $h$ of $\Bp$ is well
known due to Eichler \cite{eichler-CNF-1938}
(Deuring and Igusa gave different proofs of this result), 
and is given by 
\begin{equation}
  \label{eq:1.1}
  h=\frac{p-1}{12}+\frac{1}{3}\left
  (1-\left(\frac{-3}{p}\right )\right )+\frac{1}{4}\left
  (1-\left(\frac{-4}{p}\right )\right ),
\end{equation}
where $\left( \frac{\cdot}{p}\right ) $ denotes the Legendre
symbol. Under the correspondence $\Lambda_p\simeq \Cl(\Bp)$, 
the type number $t$ of $\Bp$ is equal to the number of non-isomorphic 
endomorphism rings of members $E$ in $\Lambda_p$.
An explicit type formula is also well known due
to Deuring~\cite{Deuring1950}, which is given by 
\begin{equation}
  \label{eq:1.2}
  t=\frac{p-1}{24}+\frac{1}{6}\left
  (1-\left(\frac{-3}{p}\right )\right )+
\begin{cases}
  h(-p)/4 & \text{if $p\equiv 1\pmod 4$},\\
  1/4+h(-p)/2 & \text{if $p\equiv 7\pmod 8$},\\
  1/4+h(-p) & \text{if $p\equiv 3\pmod 8$},\\
\end{cases}
\end{equation}
for $p>3$, and $t=1$ for $p=2,3$. Here for any square-free integer
$d\in \bbZ$, we write $h(d)$ for the class number of $\Q(\sqrt{d})$.
Though these classical results were well established by 1950,
different proofs with various ingredients such as mass formulas,
Tamagawa numbers, theta series, cusp forms, algebraic modular forms,
Atkin-Lehner involutions and traces of Hecke operators, have been
generalized and revisited many times even until now.  Different angles
and approaches such as the Eichler-Shimizu-Jacquet-Langlands
correspondence and trace formulas also play important roles in the
development.  This paper is one instance of them, where we would like
to generalize the explicit formulas (\ref{eq:1.1}) and (\ref{eq:1.2})
from $\Bp$ to $\Bp\otimes \Q(\sqrt{p})$, which is the unique totally
definite quaternion $\Q(\sqrt{p})$-algebra (up to isomorphism)
unramified at all finite places. Recall that a quaternion algebra $D$ over
a totally real number field $F$ is said to be \emph{totally definite} if
$D\otimes_{F, \sigma}\R$ is isomorphic to the Hamilton quaternion
algebra $\bbH$ for every embedding $\sigma: F\hookrightarrow \R$.
Our interest for totally definite quaternion algebras stems from the study abelian varieties over finite
fields. Note that $\Bp$ and $\Bp\otimes \Q(\sqrt{p})$ 
are the only two algebras 
that can occur as the endomorphism algebra
of an abelian variety over a finite field \cite{tate:eav}, but do
not satisfy the Eichler condition \cite[Definition~34.3]{reiner:mo}. 
The cases where the endomorphism algebras satisfy the Eicher condition
are easier to treat. Indeed, by a result of Jacobinski \cite[Theorem
2.2]{Jacobinski1968}, the class number of an order in such an
algebra is equal to a ray class number of its center. 
Observe that the number $h$ in \eqref{eq:1.1}
(resp.~$t$ in \eqref{eq:1.2}) is also equal to the number of $\Fq$-isomorphism classes
(resp.~non-isomorphic endomorphism rings) of supersingular
elliptic curves in the isogeny class corresponding to the Weil 
$q$-number $\sqrt{q}$ (or $-\sqrt{q}$), for any field $\Fq$ containing
$\F_{p^2}$. Thus, the present work may be also reviewed as an 
generalization of explicit formulas (\ref{eq:1.1}) 
and (\ref{eq:1.2}) in the arithmetic direction. A generalization of
the geometric direction in the sense that the set $\Lambda_p$ is replaced by the
superspecial locus of the Siegel moduli spaces has been inverstigated
by Hashimoto, Ibukiyama, Katsura and Oort
\cite{Hashimoto-Ibukiyama-1, Hashimoto-Ibukiyama-2,
  Hashimoto-Ibukiyama-3, Ibukiyama-Katsura-Oort-1986,
  katsura-oort:surface}. 

Now let $q$ be an odd power of $p$, and let $\Sp(\sqrt{q})$ be the set
of isomorphism classes of superspecial abelian surfaces over 
$\Fq$ in the isogeny class corresponding to the Weil number 
$\pm \sqrt{q}$. 
As a generalization of (\ref{eq:1.1}), T.-C. Yang and the present authors
\cite[Theorem 1.2]{xue-yang-yu:ECNF} (also see \cite[Theorem
1.3]{xue-yang-yu:sp_as}) show the following explicit 
formula for $|\Sp(\sqrt{q})|$.

\begin{thm}\label{1.1}
  Let $F=\Q(\sqrt{p})$, and $O_F$ be its ring of integers. 

(1) The cardinality of $\Sp(\sqrt{q})$ depends only on $p$, and is
    denoted by $H(p)$. 

(2) We have $H(p)=1,2,3$ for $p=2,3,5$, respectively.

(3) For $p>5$ and $p\equiv 3 \pmod 4$, one has
\begin{equation*}
  H(p)=h(F)\left [  \frac{\zeta_F(-1)}{2} +
     \left(13-5 \left(\frac{2}{p}\right)
    \right)\frac{h(-p)}{8}+\frac{h(-2p)}{4}+\frac{h(-3p)}{6}  \right ],
\end{equation*}
where $h(F)$ is the class number of $F$ and $\zeta_F(s)$ is the
Dedekind zeta function of $F$. 

(4) For $p>5$ and $p\equiv 1 \pmod 4$, one has
\begin{equation*}
   H(p)=       
      \begin{cases}
        h(F) \left [8 \zeta_F(-1)+ h(-p)/2+\frac{2}{3}
        h(-3p) \right ] & \text{for $p\equiv 1 \pmod 8$;} \\
      h(F) \left [\left(\frac{45+\varpi}{2\varpi}\right) \zeta_F(-1)+\left
        (\frac{9+\varpi}{8 \varpi}\right ) h(-p)
      +\frac{2}{3} h(-3p) \right ] & \text{for $p\equiv 5 \pmod 8$;} \\
      \end{cases} 
\end{equation*}
where $\varpi:=[O_F^\times: A^\times]\in\{1,3\}$ and
$A:=\Z[\sqrt{p}]$ is the suborder of index $2$ in $O_F$. 
\end{thm}

Let $\calT(\Sp(\sqrt{q}))$ denote the set of isomorphism classes of
endomorphism rings of abelian surfaces in $\Sp(\sqrt{q})$. 
The cardinality of $\calT(\Sp(\sqrt{q}))$ again depends only on the prime
$p$ (\cite[Theorem 1.3]{xue-yang-yu:sp_as},  see also Sect.~3), 
and is denoted by $T(p)$. In this paper we give an 
explicit formula for $T(p)$, which generalizes (\ref{eq:1.2}). 

\begin{thm}\label{1.2} 
Let $F=\Q(\sqrt{p})$ and $T(p):=|\calT(\Sp(\sqrt{q}))|$. 

(1) We have $T(p)=1,2,3$ for $p=2,3,5$, respectively.

(2) For $p\equiv 3 \pmod 4$ and $p\ge 7$, we have
\begin{equation}
  \label{eq:1.5}
  T(p)= \frac{\zeta_F(-1)}{2} +
     \left(13-5\left(\frac{2}{p}\right)
    \right)\frac{h(-p)}{8}+\frac{h(-2p)}{4}+\frac{h(-3p)}{6}. 
\end{equation}

(3) For $p\equiv 1 \pmod 4$ and $p\ge 7$, we have
\begin{equation}
  \label{eq:1.6}
  T(p)= 8 \zeta_F(-1) +
     \frac{1}{2}h(-p)+ \frac{2}{3} h(-3p). 
\end{equation}
\end{thm}




It follows from Theorems~\ref{1.1} and \ref{1.2} that $H(p)=T(p) h(F)$
except for the case where $p\equiv 5 \pmod 8$ and $\varpi=1$. When $p\equiv 3\pmod
4$, we actually prove this result first and use it to get formula (\ref{eq:1.5}). For the case where $p\equiv 1 \pmod 4$, we
explain how this coindence arises in part (1) of Remark~\ref{3.3}.

Similar to Deuring's result for supersingular elliptic curves, 
the proof of Theorem~\ref{1.2} is reduced to the calculation of the type
number of a maximal order $\bbO_1$ in in $\Bp\otimes \Q(\sqrt{p})$, as
well as those of certain proper $\Z[\sqrt{p}]$-orders $\bbO_{8}$ and
$\bbO_{16}$ (see (\ref{eq:bbOr1}) and (\ref{eq:bbOr2}) for definition
of these orders).  We recall briefly the concept of \emph{proper
  $A$-orders}. Henceforth all orders are assumed to be of full rank in
their respective ambient algebras. Let $F$ be a number field, and $A$ be a
$\Z$-order in $F$ (not necessarily maximal). A $\Z$-order $\calO$ in a
finite dimensional semisimple $F$-algebra $\scrD$ is said to be a
proper $A$-order if $\calO\cap F=A$. Unlike
\cite[Definition~23.1]{curtis-reiner:1}, we do not require $\calO$ to
be a projective $A$-module. The class number of $\calO$ 
is denoted by $h(\calO)$. When $F$ is a totally real number field, a
proper $A$-order in a CM-extension (i.e. a totally imaginary
quadratic extension) of $F$ is call a \emph{CM proper $A$-order}. If
$A$ coincides with $O_F$, then we drop the adjective ``proper'' and simply write ``$O_F$-order''.

It is well known that the type number of a  totally definite 
Eichler order of square-free level can be calculated by 
Eichler's trace formula \cite{eichler:crelle55}; see also \cite{korner:1987} for general $O_F$-orders. 
Some errors of Eichler's formula were found and corrected later 
independently by M.~Peters \cite{Peters1968} and by
A.~Pizer~\cite{Pizer1973}. Eichler's trace formula contains a
number of data which are generally not easy to compute. 
In \cite[p.~212]{vigneras:inv} M.-F, Vign\'eras 
gave an explicit formula for the type number
of any totally definite quaternion
algebra $D$, over any real quadratic field $F=\Q(\sqrt{d})$, that is  
unramified at all finite places. Her formula was based on 
the explicit formula for the class number $h(D)$ 
in \cite[Theorem 3.1]{vigneras:ens} and 
the class-type number relationship $h(D)=t(D)h(F)$; 
see \cite[p.~212]{vigneras:inv}.
Unfortunately, it was pointed out by 
Ponomarev in \cite[Concluding remarks, p.~103]{ponomarev:aa1981} 
that the formula in
\cite[p.~212]{vigneras:inv} is not a formula for $t(D)$ in general. 
This conclusion is based on his explicit calculations for class
numbers of positive definite quaternary quadratic forms \cite[Sect.~5]
{ponomarev:aa1981}, and the correspondence between quaternary quadratic
forms and types of quaternion algebras established in
\cite[Sect.~4]{ponomarev:aa1976}.  
The source of the difficulty is that the class-type number
relationship $h(D)=t(D)h(F)$ may fail in general\footnote{In \cite[p.~103]
{ponomarev:aa1981} it reads 
``the class number 
of $\grA_K$ divided by the proper class number of $K$ is not, in
general,  the
type number $t$''. However it should read ``the class number'' instead of
``the proper class number'' as the former one is what 
Vign\'eras' formula is based on.}, 
even if $D$ is unramified at all finite places. To remedy this, we
examine more closely the Picard group action described below.


Let $F$ be an arbitrary number field, $A\subseteq O_F$ a $\Z$-order in
$F$, and $\calO$ a proper $A$-order in a quaternion $F$-algebra $D$. The Picard group $\Pic(A)$
acts naturally on the finite set $\Cl(\calO)$ of locally
principal right  ideal classes of $\calO$ by the map
\begin{equation}
  \label{eq:8}
\Pic(A)\times
\Cl(\calO)\to \Cl(\calO), \qquad ([\gra], [I])\mapsto [\gra I], 
\end{equation}
where $\gra$ (resp.~$I$) denotes a locally principal fractional
$A$-ideal in $F$ (resp.~right $\calO$-ideal in $D$),  and $[\gra]$
(resp.~$[I]$) denotes its ideal class.  Let  $\overline{\Cl}(\calO)$ be
the set of orbits of this action, and $r(\calO)$ be its cardinality: 
\begin{equation}
  \label{eq:21}
r(\calO):=\abs{\overline{\Cl}(\calO)}=\abs{\Pic(A)\backslash
  \Cl(\calO)}.
\end{equation}
One of main results of the this paper is a formula for
$r(\calO)$ when $D$ is a totally definite quaternion algebra over a
totally real field.

  \begin{thm}\label{thm:main}
Let $A$ and $\calO$ be as above. Suppose that $F$ is a totally real number field, and $D$ is a totally definite quaternion $F$-algebra. The number of orbits of the Picard group
action (\ref{eq:8}) can be calculated by the following formula: 
\[r(\calO)=\frac{1}{h(A)}\left(h(\calO)+\sum_{1\neq [\gra]\in
      \grC_2(A, \wt A)}\ \sum_{\dbr{B}\in \scrB_{[\gra]}} \frac{1}{2}(2-\delta(B))h(B)\prod_{\ell}m_\ell(B) \right),\]
where we refer to (\ref{eq:def-grC2}), (\ref{eq:13}), (\ref{eq:46}),
(\ref{eq:35}) for the definitions of $\grC_2(A, \wt A)$, $\scrB_{[\gra]}$, $\delta_B$,  $m_\ell(B)$ respectively. 
  \end{thm}

If $A=O_F$, then the formula for $r(\calO)$ can be simplified further (see Theorem~\ref{thm:orbit-num-formula}).

Assume that $F$ is a totally real field of even degree over $\Q$, and
$D$ is the unique up to isomorphism quaternion $F$-algebra unramified
at all finite places of $F$.  Let $\calO$ be a maximal $O_F$-order
in $D$. Then $t(D)=r(\calO)$, and hence Theorem~\ref{thm:main} leads
to a type number formula for such $D$ in
Corollary~\ref{cor:type-numer-tot-def-unram}. In  \cite[Remarque, p.~82]{vigneras:ens},
Vig\'enras asserted that $h(D)/h(F)$ is always an integer.  However, the
assertion was mixed with the misconception that $h(D)=t(D)h(F)$ holds
unconditionally on $F$, and we could not locate a precise proof of this
integrality elsewhere. As an application of our orbit number formula, we prove
in Theorem~\ref{thm:integrality} that $h(D)/h(F)\in \bbN$ for all $F$. On the other hand, we  give in
Corollary~\ref{2.7}  a necessary and
sufficient condition on $F$ such that the
relationship $h(D)=h(F)t(D)$ remains valid. In particular, for
real quadratic fields satisfying this condition, 
Vig\'enras's formula \cite[p.~212]{vigneras:inv} does give a formula for the
type number $t(D)$. This approach of calculating the type number via the class number also paves the way to our proof of
Theorem~\ref{1.2}, where we treat certain orders (namely, $\bbO_8$
and $\bbO_{16}$) that does not contain
$O_F$.  The type numbers of such orders are not covered by previous
methods of Eichler-Pizer and Ponomarev.

This paper is organized as follows. In Section~\ref{sec:02} we provide
some preliminary studies
 on the $\Pic(A)$-action on $\Cl(\calO)$. This is carried out more in depth for totally definite
quaternion algebras in Section~\ref{sec:orbit-number-formula},  and we
derive the orbit number formula and its corollaries there.  The calculations
for Theorem~\ref{1.2} are worked out in
Section~\ref{sec:03}, and we prove the integrality of $h(D)/h(F)$ in
Section~\ref{sec:integrality}.

\section{Preliminaries on the Picard group action}
\label{sec:02}

Let $F$ be a number field, $O_F$ its ring of integers, and
$A\subseteq O_F$ a $\Z$-order in $F$. Let $D$ be a quaternion
$F$-algebra and $\calO$ a {\it proper} $A$-order in $D$. This
section provides a preliminary study of the $\Pic(A)$-action on
$\Cl(\calO)$ in (\ref{eq:8}).



We follow the notation of \cite[Sections
2.1--2.3]{xue-yang-yu:ECNF}. Recall that $D$ admits a canonical
involution $x\mapsto \bar x$ such that $\Tr(x)=x+\bar x$ and
$\Nr(x)=x\bar x$ are respectively the \emph{reduced trace} and \emph{reduced norm}
of $x\in D$.  The \emph{reduced discriminant} $\grd(D)$ of $D$ is the product of
all finite primes of $F$ that are ramified in $D$. 
Let $\wt A:=\Nr_A(\calO)$ be the norm of $\calO$ over
$A$. More explicitly, $\wt
A$ is the $A$-submodule of $O_F$ spanned by the reduced norms of
elements of $\calO$. Clearly,  $\wt A$ is closed under
multiplication, hence a suborder of $O_F$ containing $A$. By
\cite[Lemma~3.1.1]{xue-yang-yu:ECNF}, $\calO$ is closed under the
canonical involution if and only if $\wt A=A$. In particular, any
$O_F$-order in $D$ is closed under the canonical involution.  We have a the natural surjective map between the Picard groups:
\begin{equation}
  \label{eq:9}
  \pi:\Pic(A)\twoheadrightarrow \Pic(\wt A),\qquad \gra \mapsto \gra \wt A.
\end{equation}

Given an ideal class $[I]\in \Cl(\calO)$, we study the stabilizer
$\Stab([I])\subseteq \Pic(A)$ of the $\Pic(A)$ action on $\Cl(\calO)$ as in (\ref{eq:8}).  Let
$I^{-1}$ be the inverse of $I$, and
$\calO_l(I)=\{x\in D\mid xI\subseteq I\}$ the associated left order of
$I$. We have $\calO_l(I)=II^{-1}$ \cite[Section~I.4]{vigneras}. Thus
\begin{equation}
  \label{eq:17}
[\gra I] =[I]\quad \text{if and only if}\quad [\gra \calO_l(I)]
=[\calO_l(I)],   
\end{equation}
where $[\calO_l(I)]$ is the trivial ideal class of
$\calO_l(I)$. Therefore, the study of $\Stab([I])$ often reduces to that of
$\Stab([\calO])$. We have 
\begin{equation}
  \label{eq:31}
\Stab([\calO])=\{ [\gra]\in \Pic(A)\mid \exists \lambda\in D \text{ such that }
\gra\calO=\lambda\calO\}.  
\end{equation}
Since $\gra \calO$ is a nonzero 
two-sided $\calO$-ideal,   $\lambda$ lies in  the normalizer
$\calN(\calO)\subseteq D^\times$ by
\cite[Exercise~I.4.6]{vigneras}. To study more closely the relationship between $\gra$ and $\lambda$, we
make use of the following lemma from commutative algebra. 

\begin{lem}\label{lem:eq-unit-ideal}
Let $R\subseteq S$ be an extension of unital rings, with $R$
commutative and $S$ a finite $R$-module. 
\begin{enumerate}[(i)]
\item If $\grc\subseteq R$ is an $R$-ideal with $\grc S=S$,
  then $\grc =R$.
\item Let $L$ be the totally
quotient ring of $R$. 
 Suppose  that the natural map $S\to S\otimes_R L$ is injective, 
and $S\cap L=R$ in $S\otimes_R L$. If $\grc\subset L$ is an 
$R$-submodule with $\grc S=\lambda S$ for some $\lambda
  \in L^\times$, then $\grc = \lambda R$. 
\end{enumerate}
\end{lem}
\begin{proof}(i) 
Since $S$ is a finite $R$-module,  the equality $\grc S=S$
  implies  that there exists $a\in \grc$  such that $(1-a)S=0$
 by \cite[Corollary~4.7]{Eisenbud-Com-alg}. Necessarily $a=1$ since $1\in S$, 
 and hence $1\in \grc$ and $\grc=R$. \\
(ii) Let $\grc'=\frac{1}{\lambda} \grc$. Then $\grc'S=S$, which implies
that $\grc' \subset S$. We have $\grc' \subseteq S\cap L=R$, and hence
$\grc'$ is an integral ideal of $R$.  Now it follows from part~(i)
that $\grc'=R$,  equivalently, $\grc= \lambda R$. 
\end{proof}

We return to the study of the $\Pic(A)$-action on $\Cl(\calO)$. 
\begin{cor}\label{endo.1} 
  Suppose that $\gra\subset F$ is a locally principal nonzero fractional
  $A$-ideal with $\gra \calO=\lambda \calO$ for some $\lambda \in
  D$.  
\begin{enumerate}[(i)]
\item If $\lambda\in F$, then $\gra=\lambda  A$. 
\item If $\lambda\not \in F$, then  $\gra  B=\lambda  B$ with
  $B=F[\lambda]\cap \calO$.  In particular, $[\gra]$ belongs to the
  kernel of the canonical map $\Pic(A)\to \Pic(B)$. 
\end{enumerate}
\end{cor}  

\begin{proof}
 Part~(i) follows directly from Lemma~\ref{lem:eq-unit-ideal}(ii) with
$R=A$, $S=\calO$ and $\grc=\gra$, and part~(ii)
follows with $R=B$, $S=\calO$, and $\grc=\gra B$. 
\end{proof}

We say a locally principal fractional ideal $\gra$ of $A$ \emph{capitulates} in $B$ if the extended ideal $\gra B$ is principal. The capitulation problems (for abelian extensions of number fields $K/F$ with $A=O_F$ and $B=O_K$) was studied by Hilbert (see Hilbert Theorem~94), and it continues to be a field of active research up to this day \cite{Iwasawa-capitulation-1989, Bond-2017, suzuki-2001}.  
We will follow up on this line of investigation in
Section~\ref{sec:orbit-number-formula}, particularly in the derivation
of the orbit number formula (see Theorem~\ref{thm:orbit-num-formula}). However, our result does not explicitly depend on the works just cited.

\begin{lemma}\label{endo.2} 
For any ideal class $[I]\in \Cl(\calO)$, the stabilizer $\Stab([I])$ is contained in the kernel of the
homomorphism
\begin{equation}
  \label{eq:12}
\Sq:\Pic(A)\to \Pic(\wt A), \qquad [\gra]\mapsto [\,\gra\wt A\,]^2.   
\end{equation}
\end{lemma}

\begin{proof}
Clearly, $\Nr_A$ commutes with any
localization of $A$. Thus $\Nr_A(\gra
  \calO')=\gra^2\Nr_A(\calO')$ for every proper $A$-order $\calO'$ and \textit{locally principal} fractional
  $A$-ideal $\gra$.    Since $I$ is
 locally principal,  $\calO_l(I)$ is a proper $A$-order with
  $\Nr_A(\calO_l(I))=\wt A$ by
  \cite[Section~2.3]{xue-yang-yu:ECNF}.  Suppose that $ \gra
\calO_l(I) =\lambda \calO_l(I)$ for some $\lambda \in D^\times$. Taking $\Nr_A$ on both sides, 
we get 
\begin{equation}\label{eq:square-norm-lambda}
\gra^2 \wt A= \Nr(\lambda) \wt A,
\end{equation}
and hence $[\gra]\in \ker(\Sq)$. 
\end{proof}

\begin{cor}\label{endo.3}
Let $\Pic(\wt A)^2$ be the subgroup of $\Pic(\wt A)$ consisting of 
the $\wt A$-ideal classes that are 
perfect squares.  Then the number $h(\calO)/|\Pic(\wt A)^2|$ 
is an integer.
\end{cor}
\begin{proof}
  The map $\Sq: \Pic(A)\to \Pic(\wt A)$ factors as
  $\Pic(A)\twoheadrightarrow \Pic(\wt A)^2\hookrightarrow \Pic(\wt
  A)$.
  By Lemma~\ref{endo.2}, each $\Pic(A)$-orbit of $\Cl(\calO)$ has
  cardinality divisible by $|\Pic(\wt A)^2|$. 
  The corollary follows.
\end{proof}


\begin{cor}\label{endo.4}
  Suppose that $\calO$ is closed under the canonical involution and
  $h(A)$ is odd.  Then the action of
  $\Pic(A)$ on $\Cl(\calO)$ is free. 
\end{cor}
\begin{proof}
As remarked before, $\calO$ is closed under the canonical
involution if and only if $\wt A=A$. The two conditions imply that
$\ker (\Sq)$ is trivial, so the corollary follows from Lemma~\ref{endo.2}. 
\end{proof}

\begin{rem}
See \cite[Corollary (18.4), p.~134]{Conner-Hurrelbrink} for the complete list of
all quadratic fields $F=\Q(\sqrt{d})$ with odd class number.  
When $A\neq \wt A$, the action of $\Pic(A)$ on $\Cl(\calO)$ needs not
to be free, even if $h(A)$ is odd.  For example, let $A=\Z[\sqrt{37}]$ and $\calO=\bbO_8$ 
(see (\ref{eq:bbOr1}) and (\ref{eq:bbOr2}) for its definition).  We have $\wt A=O_F$,
$h(A)=3$ and $h(\calO)=7$ by \cite[(6.10)]{xue-yang-yu:ECNF}. Clearly, the
$\Pic(A)$-action on $\Cl(\calO)$ cannot be free. 
\end{rem}

In some cases, the Picard group action (\ref{eq:8}) can be utilized
to calculated type numbers of certain orders. Concrete examples will be worked
out in Section~\ref{sec:03}. To explain the ideas, we adopt the adelic
point of view. Let
$\wh \Z:=\varprojlim \zmod{n}=\prod_\ell \Z_\ell$ be the profinite
completion of $\Z$, where the product runs over all primes
$\ell\in \bbN$. Given any finite dimensional $\Q$-vector space or
finitely generated $\Z$-module $M$, we set
$\wh M:= M\otimes_\Z \wh\Z$, and $M_\ell:=M\otimes_\Z\Z_\ell$. Two
orders $\calO_1$ and $\calO_2$ in the quaternion $F$-algebra $D$ are said to be in the same \emph{genus}
if $\wh\calO_1\simeq \wh \calO_2$, or equivalently, if there exists
$x\in \wh D^\times$ such that $\wh\calO_1=x\wh\calO_2 x^{-1}$. They
are said to be of the same \emph{type} if $\calO_2=x\calO_1 x^{-1}$
for some $x\in D^\times$.  Let $\calT(\calO)$ denote the set of
$D^\times$-conjugacy classes of orders in the genus of $\calO$, and
let $t(\calO):=|\calT(\calO)|$ denote the \emph{type number} of
$\calO$. For example, it is well known that all maximal $O_F$-orders
of $D$ lie in the same genus.  If $\calO$ is a maximal $O_F$-order,
then the set $\calT(\calO)$ (resp.~its cardinality $t(\calO)$) depends only on $D$ and is denoted by
$\calT(D)$ (resp.~$t(D)$) instead.

Let $\calN(\wh\calO)\subseteq \wh D^\times$ be the normalizer of
$\wh\calO$, which admits a filtration $\calN(\wh \calO)\unrhd \wh
F^\times\wh \calO^\times\unrhd \wh\calO^\times$. 
 It is easy to see that $\Pic(A)\simeq \wh F^\times/F^\times \wh A^\times$,
 and 
\begin{align}
\Cl(\calO)&\simeq D^\times\backslash \wh D^\times /\wh
\calO^\times, \\
\overline  \Cl(\calO)&\simeq D^\times\backslash \wh D^\times /\wh F^\times\wh
\calO^\times, \label{eq:36}\\
 \calT(\calO)&\simeq D^\times\backslash \wh D^\times /\calN(\wh
\calO).  \label{eq:37}
\end{align}
Denote the normal subgroup $\calN(\calO)\cap (\wh F^\times\wh\calO^\times)$
of $\calN(\calO)$  by $\calN_0(\calO)$.


\begin{lem}\label{lem:normalizer-stablizer}
Keep the notation of (\ref{eq:31}). 
There is a canonical isomorphism 
\[\Stab([I])\simeq \calN_0(\calO_l(I))/F^\times\calO_l(I)^\times,
\qquad [\gra]\leftrightarrow (\lambda \bmod F^\times\calO_l(I)^\times). \]

\end{lem}
\begin{proof}
In light of (\ref{eq:17}), it is enough to show that  $\Stab([\calO])\simeq
\calN_0(\calO)/F^\times\calO^\times$, where $[\calO]\in \Cl(\calO)$ is
the  principal right $\calO$-ideal class. 
Applying the Zassenhaus lemma \cite[Lemma~5.49]{Rotman-alg} to
$F^\times\unlhd D^\times$ and $\wh \calO^\times \unlhd \wh
F^\times\wh\calO^\times$, we get 
\[
\begin{split}
\frac{\calN_0(\calO)}{F^\times\calO^\times}&=\frac{(D^\times\cap \wh
F^\times\wh\calO^\times)F^\times}{(D^\times\cap \wh
\calO^\times)F^\times}\simeq\frac{(D^\times\cap \wh F^\times \wh
\calO^\times)\wh \calO^\times}{(F^\times\cap \wh F^\times\wh
\calO^\times)\wh \calO^\times}=\frac{(D^\times\wh\calO^\times\cap \wh
F^\times)\wh\calO^\times}{F^\times\wh\calO^\times}\\
  &\simeq\frac{D^\times\wh\calO^\times\cap \wh
F^\times}{(D^\times\wh\calO^\times\cap \wh
F^\times)\cap F^\times\wh\calO^\times}= \frac{D^\times\wh\calO^\times\cap \wh
F^\times}{F^\times(\wh F^\times\cap \wh \calO^\times)}=\frac{D^\times\wh\calO^\times\cap \wh
F^\times}{F^\times\wh A^\times}=\Stab([\calO]). 
\end{split}
\]
We leave it as a routine exercise to check that the adelic isomorphism
constructed above matches with the concrete one given in the lemma. 
\end{proof}





There is a natural surjective map 
\begin{equation}
  \label{eq:32}
  \Upsilon: \Cl(\calO)\twoheadrightarrow \calT(\calO), \qquad [I]\mapsto
  \dbr{\calO_l(I)},  
\end{equation}
where $\dbr{\calO_l(I)}$ denotes the $A$-isomorphism class of
$\calO_l(I)$ (equivalently, the $D^\times$-conjugacy class of
$\calO_l(I)$).  Clearly, $\Upsilon$ factors through the projection $\Cl(\calO)\to \overline \Cl(\calO)$. 

\begin{prop}\label{prop:class-type-number}
If $\calN(\wh \calO)=\wh F^\times\wh\calO^\times$, then $t(\calO)=r(\calO)$. 
If further $\calO$ is stable under the canonical involution and
$h(A)$ is odd, then
\begin{equation}
  \label{eq:39}
 h(\calO)=h(A)t(\calO).  
\end{equation}
Moreover, for any
order $\calO'$ in the genus of $\calO$, we have 
\begin{equation}
  \label{eq:40}
\abs{\Upsilon^{-1}(\dbr{\calO'})}=h(A),  \quad \text{and} \quad
\calN(\calO')=F^\times\calO'^\times. 
\end{equation}
\end{prop}




\begin{proof}
The assumption $\calN(\wh \calO)=\wh F^\times\wh\calO^\times$ allows the
canonical identification $\overline{\Cl}(\calO)=\calT(\calO)$,
and hence identifications of the orbits of the $\Pic(A)$-action with
the fibers of $\Upsilon$. In particular, we have $t(\calO)=r(\calO)$.    By
  definition, $\wh \calO'\simeq \wh \calO$ for any order $\calO'$ in
  the genus of $\calO$. Hence $\calN(\calO')=D^\times\cap \calN(\wh\calO')=\calN_0(\calO')$.

  Suppose further that $\calO$ is stable under the canonical
  involution and $h(A)$ is odd. By Corollary~\ref{endo.4}, $\Pic(A)$
  acts freely on $\Cl(\calO)$, so we have
  $t(\calO)=r(\calO)=h(\calO)/h(A)$, and 
$\abs{\Upsilon^{-1}(\dbr{\calO'})}=h(A)$ for each $\dbr{\calO'}\in \calT(\calO)$.
The equality $\calN(\calO')=F^\times\calO'^\times$ follows from 
 Lemma~\ref{lem:normalizer-stablizer}. 
\end{proof}


It is well known that the condition $\calN(\wh \calO)=\wh
F^\times\wh\calO^\times$ holds if $D$ is unramified at all the finite
places of $F$ and $\calO$ is a maximal $O_F$-order (see Corollary~\ref{2.7}). We provide a class of non-maximal
orders (namely, $\bbO_{16}$ defined by (\ref{eq:bbOr1}) and (\ref{eq:bbOr2})) that satisfies this condition in
Section~\ref{sec:03}.


\section{Picard group action for totally definite quaternion algebras}
\label{sec:orbit-number-formula}

Throughout this section, we assume that $F$ is a totally real field,
and $D$ is a totally definite quaternion $F$-algebra. In particular, for any $\lambda \in D$ with
$\lambda \not \in F$, the $F$-algebra $F[\lambda]$ is a CM-extension
of $F$.  Let $A$ be a $\Z$-order in $F$, and
$\calO$ be a proper $A$-order in $D$. The main goal of this section is
to derive a orbit number formula for the $\Pic(A)$ action on
$\Cl(\calO)$. This also leads to a type number formula for
$t(D)$ when $F$ has even degree over $\Q$ and $D$ is unramified at
all finite places of $F$.





\begin{notn}
Let $R_1\subseteq R_2$ be an extension of commutative Noetherian rings. The
canonical homomorphism between the Picard groups is denoted by 
\[i_{R_2/R_1}:\Pic(R_1) \to \Pic(R_2), \qquad   [\gra]\mapsto [\gra
R_2].\]
If $F_2/F_1$ is a finite extension of number fields, we write
$i_{F_2/F_1}$ for $i_{O_{F_2}/O_{F_1}}$. Because of 
Corollary~\ref{endo.1}, we are interested in the following set of
isomorphism classes of CM proper $A$-orders: 
\begin{equation}
  \label{eq:41}
\scrB:=\{\dbr{B}\mid B \text{ is a CM proper $A$-order, and }
\ker(i_{B/A})\neq \{1\}\},
\end{equation}
where $\dbr{B}$ denotes the $A$-isomorphism class of $B$. 
\end{notn}

\begin{prop}\label{prop:finiteness-CM-proper-A-order}
The  set $\scrB$ is finite.  
\end{prop}

\begin{proof}
Let $B$ be a CM proper $A$-order with $\dbr{B}\in \scrB$ and $K$ be its
fractional field.  Pick  a  nontrivial
ideal class $[\gra]\in \ker(i_{B/A})$ so that  
\begin{equation}
  \label{eq:11}
\gra B=\lambda B \quad \text{with}\quad \lambda \in K^\times. 
\end{equation}
  Necessarily
$\lambda \not\in F^\times$, otherwise $\gra=\lambda A$ by part~(ii) of
Lemma~\ref{lem:eq-unit-ideal}. 

There is a commutative diagram
\[
\begin{CD}
  \Pic(A) @>>> \Pic(O_F)\\
   @VVV  @VVV\\
  \Pic(B)@>>> \Pic(O_K)
\end{CD}
\]
If $[\gra]\not\in\ker(i_{O_F/A})$, then $[\gra O_F]$ is a nontrivial
ideal class in $\ker(i_{K/F})$. By
\cite[Section~14]{Conner-Hurrelbrink}, there are only finitely many
CM-extensions $K'/F$ such that $\ker(i_{K'/F})\neq \{1\}$. On the
other hand, suppose that $[\gra]\in \ker(i_{O_F/A})$, so that 
$\gra O_F=\gamma O_F$ for some $\gamma \in F^\times$.  We have
$\gamma O_K=\gra O_K=\lambda O_K$, and hence $\lambda=\gamma u$
for some $u\in O_K^\times$. Note that $u\not\in O_F^\times$ since
$\lambda \not\in F^\times$.  There
are only finitely many CM-extensions $K''/F$ such that
$[O_{K''}^\times:O_F^\times]>1$. Indeed, let $\bmu(K'')$ be the group
of roots of unity in $K''$. There are only finitely many CM-extensions $K''/F$ such that
$\bmu(K'')\neq \{\pm 1\}$. If $\bmu(K'')=\{\pm 1\}$, then
$[O_{K''}^\times: O_F^\times]>1$ if and only if $K''/F$ is a
CM-extension of type II (i.e.  $[O_{K''}^\times: O_F^\times\bmu(K'')]=2$), and there are finitely many of these by
\cite[Lemma~13.3]{Conner-Hurrelbrink}. 
We conclude that the following set of
CM-extensions of $F$ is finite:
\[\scrK:=\{B\otimes_A F\mid \dbr{B}\in \scrB \}.\]

Now fix $K\in \scrK$ and a nontrivial $A$-ideal class
$[\gra]\in \ker(i_{O_K/A})$. Pick $\lambda'\in K^\times$ such that
$\gra O_K=\lambda' O_K$. The CM proper $A$-orders $B\subseteq O_K$
such that $[\gra]\in \ker(i_{B/A})$ are characterized in
Lemma~\ref{lem:char-CM-order-nontrivial-i} below.  For each
$u\in O_K^\times$ with $u^{-1}\lambda'\not\in F^\times$, we have
$A\cap \frac{u\gra}{\lambda'}=\{0\}$, so $A+\frac{u\gra}{\lambda'}$
forms an $A$-lattice in $O_K$. There are only finitely many $A$-orders
in $O_K$ containing $A\oplus \frac{u\gra}{\lambda'}$. On the other hand,
there is one-to-one correspondence between $O_K^\times/A^\times$ and
the set of $A$-submodules of the form $u\gra/\lambda'$ in $O_K$. Since
$K/F$ is a CM-extension, $O_K^\times/A^\times$ is a finite
group. Therefore,  there are only finitely many CM proper $A$-order $B$ in
$K$ with $[\gra]\in \ker(i_{B/A})$.  Since $\ker(i_{O_K/A})$ is finite
as well, the proposition follows. 
\end{proof}




\begin{lem}\label{lem:char-CM-order-nontrivial-i}
  Let $K/F$ be a CM-extension with $\ker(i_{O_K/A})\neq \{1\}$, and
  $[\gra]$ be a nontrivial $A$-ideal class in $\ker(i_{O_K/A})$ so
  that $\gra O_K=\lambda' O_K$ for some $\lambda'\in K^\times$. Given
  a CM proper $A$-order $B\subseteq O_K$, we have
  $[\gra]\in \ker(i_{B/A})$ if and only if there exists
  $u\in O_K^\times$ such that\footnote{Up to multiplication by a unit in $O_K^\times$, 
    the $A$-submodule $\gra/\lambda'\subset O_K$ depends only on the $A$-ideal
    class $[\gra]\in \ker(i_{O_K/A})$ and not on the choice of $\gra$.}
  $u\gra/\lambda'\subseteq B$.  In addition, we always have  $u^{-1}\lambda'\not\in F^\times$ if such a unit $u\in O_K^\times$ exists. 
\end{lem}
\begin{proof}
First, suppose that $[\gra]\in \ker(i_{B/A})$ so that (\ref{eq:11})
holds. We have 
$\lambda' O_K=\gra O_K=\lambda O_K$,  
 and hence
$\lambda'=u\lambda$ for some $u\in O_K^\times$. It follows that 
\[u\gra/\lambda'=\gra/\lambda\subset B.\]
As remarked right after (\ref{eq:11}),   $u^{-1}\lambda'=\lambda\not\in F^\times$ since $[\gra]$ is assumed to be nontrivial. 

   Conversely, suppose that there exists $u\in O_K^\times$ such that
  $u\gra/\lambda'\subset B$. Let $\grb=\frac{u\gra}{\lambda'}B$, the
  ideal of $B$ generated by $ u\gra/\lambda'$. Then
  $\grb O_K=\frac{u\gra}{\lambda'} O_K=O_K$.  It follows from
  Lemma~\ref{lem:eq-unit-ideal} that $\grb=B$, i.e.
  $\gra B=(u^{-1}\lambda') B$.
\end{proof}
When $A$ coincides with the maximal order $O_F$ in $F$, we give a more
precise characterization of $O_F$-orders $B$ in $K$ with $\ker(i_{B/O_F})\neq
\{1\}$ in Lemma~\ref{lem:suborder-nontrivial-kernel}.

Let $B$ be an arbitrary CM proper $A$-order with fractional field $K$, and $x
\mapsto \bar{x}$ be the unique
  nontrivial involution of $K/F$. We define a symbol 
\begin{equation}\label{eq:46}
  \delta(B):=
  \begin{cases}
    1\qquad & \text{if } B=\bar B,\\
    0\qquad & \text{otherwise.}
  \end{cases}
\end{equation}
Note that $\delta(B)=1$
  for all $O_F$-orders $B$. If $\calO$ is stable under the canonical involution of $D$, then so is
$\calO_l(I)$ for every locally principal right $\calO$-ideal
$I$, in which case  $\delta(F(\lambda)\cap
\calO_l(I))=1$ for every $\lambda\in D^\times$ with
$\lambda \not \in F^\times$. 


\begin{lem}\label{lem:kernel-2-torsion}
If $B$ is a CM proper
  $A$-order with $\delta(B)=1$, then $\abs{\ker (i_{B/A})}\leq 2$.
\end{lem}
\begin{proof}
  This result is well known when $A=O_F$ and $B=O_K$. In fact, the proof of 
\cite[Theorem~10.3]{Washington-cyclotomic} applies,  \textit{mutatis
  mutandis},  to the current setting as well and shows that  $\abs{\ker
  (i_{B/A})}\leq 2$. 
\end{proof}



Lemma~\ref{lem:kernel-2-torsion} does not
  hold in general when $\delta(B)=0$. We exhibit in Example~\ref{ex:ker-cyclic-order-3} an infinite family of pairs $(A, B)$
  such that $\ker(i_{B/A})\simeq \zmod{3}$.

 Denote by $\Emb(B,\calO)$ the set of optimal embeddings from $B$ into $\calO$:
\[\Emb(B, \calO)=\{\varphi\in \Hom_F(K, D)\mid \varphi(K)\cap
\calO=\varphi(B)\}.  \]
Fix a \emph{nontrivial} $A$-ideal class $[\gra]\in \Pic(A)$. We define a few
subsets of $\scrB$:
\begin{align}
  \scrB(\calO)&=\{\dbr{B}\in \scrB\mid \Emb(B, \calO)\neq
\emptyset\},\\
  \scrB_{[\gra]}&=\{\dbr{B}\in \scrB\mid [\gra]\in \ker(i_{B/A})\}, \label{eq:13}\\
\scrB_{[\gra]}(\calO)&=\{\dbr{B}\in \scrB\mid [\gra]\in
                   \ker(i_{B/A}),\text{ and }\Emb(B, \calO)\neq
\emptyset\}.\label{eq:19}
\end{align}
Corollary~\ref{endo.1} implies that for any ideal class $[I]\in \Cl(\calO)$, 
\begin{equation}
  \label{eq:stabilizer}
  \Stab([I])=\bigcup_{\dbr{B}\in \scrB(\calO_l(I))} \ker(i_{B/A}:\Pic(A) \to \Pic(B)). 
\end{equation}

Suppose that $\scrB_{[\gra]}(\calO_l(I))\neq \emptyset $ for some ideal class $[I]\in \Cl(\calO)$ so that $[\gra]\in \Stab([I])$. It then follows from \eqref{eq:square-norm-lambda} that $\gra^2\wt A=\Nr(\lambda)\wt A$ for some $\lambda \in D^\times$, where $\wt A=\Nr_A(\calO)$ is the suborder of $O_F$ spanned over $A$ by the reduced norms of elements of $\calO$.   
Note that $\Nr(\lambda)$ is totally positive since $D$ is totally definite over $F$. For any $\Z$-order $A'$ in $F$, Let $\Pic^+(A')$ be the quotient of the multiplicative group of locally principal fractional $A'$-ideals by the subgroup of principal fractional $A'$-ideals that are generated by \emph{totally positive} elements. If $A'=O_F$, then $\Pic^+(A')$ is simply the narrow class group of $F$.  We denote the subgroup of 2-torsions of $\Pic^+(A')$ by $\Pic^+(A')[2]$, and the canonical map $\Pic^+(A')\to \Pic(A')$ by $\theta_{A'}$. For any order $A'$ intermediate to $A\subseteq O_F$,  let us define
\begin{equation}\label{eq:def-grC2}
\grC_2(A, A'):=\{[\gra]\in \Pic(A)\mid i_{A'/A}([\gra])\in \theta_{A'}(\Pic^+(A')[2])\}.
\end{equation}
If $A=O_F$, then necessarily $A'=O_F$ as well, and we simply write $\grC_2(F)$ for $\grC_2(O_F, O_F)$, which coincides with the image of $\Pic^+(O_F)[2]$ in $\Pic(O_F)$.
Clearly, $\grC_2(A, \wt A)$ is a subgroup of $\ker(\Sq)$, where $\Sq: \Pic(A)\to \Pic(\wt A)$ is the homomorphism defined in \eqref{eq:12}.  By our construction, 
\begin{equation}\label{eq:tot-def-not-stab}
[\gra]\not\in  \grC_2(A, \wt A)\quad \text{implies that}\quad  [\gra] \not\in \Stab([I]), \ \forall [I]\in \Cl(\calO). 
\end{equation}

Using the norm map $\Nm_{K/F}$ for CM-extensions $K/F$, one shows similarly that $[\gra]\in
\grC_2(A, O_F)$ if 
$\scrB_{[\gra]}\neq\emptyset$. 
We claim that this is in fact an ``if and only if" condition when $A=O_F$. If $\gra$ is a
non-principal fractional 
$O_F$-ideal such that  $\gra^2=\tau O_F$ for some totally positive $\tau\in F^\times$, then  $K=F(\sqrt{-\tau})$ is a CM-extension of $F$ with $[\gra]\in
\ker(i_{K/F})$ (cf.~\cite[Theorem~14.2]{Conner-Hurrelbrink}). Thus 
$\scrB_{[\gra]}\neq \emptyset $ if and only if $[\gra]\in \grC_2(F)$.
Combining this with  Lemma~\ref{lem:kernel-2-torsion}, we obtain that 
\begin{equation}
  \label{eq:30}
  \scrB=\coprod_{1\neq [\gra]\in \grC_2(F)}\scrB_{[\gra]}\qquad\text{if }
 A=O_F.
\end{equation}
Therefore, $\scrB=\emptyset$ if and only if $\grC_2(F)$ is
trivial.





\begin{cor}\label{2.7}
  Let $F$ be a totally real number field of even degree over $\Q$, and $\calO$
  be a maximal $O_F$-order in the (unique up to isomorphism)   totally definite 
  quaternion $F$-algebra $D$ 
  unramified at all finite places of $F$.  
  The following are equivalent:

(1) the group $\grC_2(F)$ is trivial; 


(2) the action of $\Pic(O_F)$ on $\Cl(\calO)$ is free; 

(3) $t(D)=h(D)/h(F)$. 

\noindent Moreover, if $h(F)$ is odd, then the statements (1)--(3) hold. 
\end{cor}
\begin{proof}
By \cite[Section~II.2]{vigneras},   we have  $\calN(\wh \calO^\times)=\wh F^\times \wh
  \calO^\times$ in this case, and hence $t(\calO)=r(\calO)\geq
  h(D)/h(F)$.  The equality $r(\calO)=
  h(D)/h(F)$ holds  if and only if the action of
  $\Pic(O_F)$ on $\Cl(\calO)$ is free, establishing the equivalence
  between (2) and (3).

  For the equivalence between (1) and (2), it is enough to show that
  the latter holds if and only if $\scrB=\emptyset$.  The sufficiency
  is clear by (\ref{eq:stabilizer}). As $D$ is unramified at
  all finite places, any CM extension $K$ of $F$ can be embedded into
  $D$ by the local-global principle. If
  $\scrB\neq \emptyset$, then there exists a CM-extension $K$ of $F$
  such that $\ker(i_{K/F})\neq \{1\}$. Fix an $F$-embedding $\varphi$
  from such a $K$ into $D$.  The image $\varphi(O_K)$ is contained in
  a maximal $O_F$-order $\calO'$. Set $I:=\calO'\calO$, which is a
  fractional right $\calO$-ideal with $\calO_l(I)=\calO'$. We have
  $\dbr{O_K}\in \scrB(\calO_l(I))$, and hence
  $\ker(i_{K/F})\subseteq \Stab([I])$ by (\ref{eq:stabilizer}).
  Therefore, $\Pic(O_F)$ acts freely on $\Cl(\calO)$ if and only if
  $\scrB=\emptyset$.
\end{proof}







It is possible that $\grC_2(F)$ is trivial despite $h(F)$ being
even. Indeed, one of such examples is given by $F=\Q(\sqrt{34})$ (see 
\cite[Section~12, p.~64]{Conner-Hurrelbrink}).  We provide two
applications of Corollary~\ref{2.7}. First, if $F$ is a real
quadratic field satisfying condition (1) of Corollary~\ref{2.7} and
$D$ is a totally definite quaternion $F$-algebra unramified at all
finite places of $F$, then an explicit formula for $t(D)$ of $D$ is given in \cite[p.~212]{vigneras:inv}.

Second, recall that $\calO^\times/O_F^\times$ is a finite 
group \cite[Theorem~V.1.2]{vigneras}. For any finite group $G$, we
define
\begin{align}
 h(D, G)&:=\abs{\{[I]\in \Cl(\calO)\mid \calO_l(I)/A^\times\simeq
  G\}},\\ t(D, G)&:=\abs{\{\dbr{\calO'}\in \calT(D)\mid \calO'/A^\times\simeq
  G\}}.
\end{align}
Here $h(D, G)$ does
not depends on the choice of the maximal order $\calO$, hence
denoted as such. 
The following result is used in \cite{li-xue-yu:unit-gp}. 
\begin{cor}
Keep the assumptions of Corollary~\ref{2.7} and suppose that $\grC_2(F)$
is trivial.  Then $h(D, G)=h(F)
t(D, G)$ for any finite group $G$. 
\end{cor}
\begin{proof}
The assumptions guarantee that (\ref{eq:40}) holds in the current
setting as well. The corollary follows immediately. 
\end{proof}

We return to the general setting of this section, that is,  $\calO$ being a
proper $A$-order.   Following \cite[Section~3]{xue-yang-yu:ECNF}, we set up the
  following notation
  \begin{align*}
    w(\calO)&:=[\calO^\times: A^\times], & w(B)&:=[B^\times:
    A^\times],\\
    m(B, \calO)&:=\abs{\Emb(B, \calO)}, & m(B, \calO, \calO^\times)&:=\abs{\Emb(B, \calO)/\calO^\times},
  \end{align*}
where $\calO^\times$ acts on $\Emb(B, \calO)$ from the right by
$\varphi\mapsto g^{-1}\varphi g$ for all $\varphi\in \Emb(B,\calO)$
and $g\in \calO^\times$. Note that we have 
\begin{equation}
  \label{eq:25}
  m(B, \calO)=m(B, \calO, \calO^\times)w(\calO)/w(B).
\end{equation}
 For each prime $\ell\in \bbN$, let $B_\ell$
(resp.~$\calO_\ell$) be
the $\ell$-adic completion of $B$ (resp.~$\calO$). For simplicity, we
put 
\begin{equation}\label{eq:35}
m_\ell(B):=m(B_\ell, \calO_\ell,
\calO_\ell^\times)=\abs{\Emb(B_\ell, \calO_\ell)/\calO_\ell^\times}.
\end{equation}
It is well known \cite[Theorem~II.3.2]{vigneras} that $m_\ell(B)=1$ for almost all prime $\ell\in
\bbN$.

  \begin{thm}\label{thm:orbit-num-formula}
Let $A$ be a $\Z$-order in a totally real number field $F$, and
$\calO$ be a proper $A$-order in a totally
definite quaternion $F$-algebra $D$.  The number of orbits of the Picard group
action (\ref{eq:8}) can be calculated by the following formula: 
\begin{equation*}\label{eq:48}
r(\calO)=\frac{1}{h(A)}\left(h(\calO)+\sum_{1\neq [\gra]\in
      \grC_2(A, \wt A)}\ \sum_{\dbr{B}\in \scrB_{[\gra]}} \frac{1}{2}(2-\delta(B))h(B)\prod_{\ell}m_\ell(B) \right),
\end{equation*}
where $\grC_2(A, \wt A)$ is the subgroup of $\Pic(A)$ defined in (\ref{eq:def-grC2}), and $\scrB_{[\gra]}$ is the
set of isomorphism classes of CM proper $A$-orders defined in
(\ref{eq:13}).   Moreover, if $A=O_F$, then 
  \begin{equation}\label{eq:49}
r(\calO)=\frac{1}{h(F)}\left(h(\calO)+\frac{1}{2}\sum_{\dbr{B}\in \scrB} h(B)\prod_{\ell}m_\ell(B) \right),    
  \end{equation}
where $\scrB$ is the set of isomorphism classes of CM $O_F$-orders in (\ref{eq:41})
with $A=O_F$. 
  \end{thm}

The proof of this theorem relies on Burnside's lemma
\cite[Theorem~2.113]{Rotman-alg}, which we recall briefly for the
convenience of the reader. 

\begin{lem}[Burnside's lemma]
  Let $G$ be a finite group acting on a finite set $X$, and $r$ be the
  number of orbits. For each $g\in
  G$, let $X^g\subseteq X$ be the subset of elements fixed by $g$. Then  $r=\frac{1}{G}\sum_{g\in G}\abs{X^g}$. 
\end{lem}

Hence we reduce the proof of Theorem~\ref{thm:orbit-num-formula} to the calculation of
$\abs{\Cl(\calO)^{[\gra]}}$ for  each
$[\gra]\in \Pic(A)$, where $\Cl(\calO)^{[\gra]}\subseteq \Cl(\calO)$ denotes the subset of locally
   principal right $\calO$-ideal classes fixed by $[\gra]$.    By
  \eqref{eq:tot-def-not-stab}, if $[\gra]\not\in \grC_2(A, \wt A)$, then
  $\Cl(\calO)^{[\gra]}=\emptyset$.




\begin{prop} \label{prop:fixed-number}
For each nontrivial $A$-ideal class $[\gra]\in \grC_2(A, \wt A)$,
  \begin{equation}
    \label{eq:14}
\abs{\Cl(\calO)^{[\gra]}}=    \sum_{\dbr{B}\in \scrB_{[\gra]}}
\frac{1}{2}(2-\delta(B))h(B)\prod_{\ell}m_\ell(B). 
  \end{equation}
\end{prop}
\begin{proof}
  We fix a complete set of representatives
  $\{I_j\mid 1\leq j \leq h:=h(\calO)\}$ of the ideal classes in
  $\Cl(\calO)$,  and define $\calO_j:=\calO_l(I_j)$. For each $1\leq
  j\leq h$, consider the
  following two sets
  \begin{equation}\label{eq:20}
    X(\calO_j):=\{\lambda \in D^\times\mid \gra \calO_j=\lambda
                \calO_j\},\qquad 
    Y(\calO_j):=X(\calO_j)/A^\times,
  \end{equation}
where $A^\times$ acts on $X(\calO_j)$ by multiplication.  Note that
$X(\calO_j)\cap F=\emptyset$ since $\gra$ is assume to be  non-principal. By
(\ref{eq:17}), we have 
\begin{equation}
  \label{eq:16}
    X(\calO_j)\neq \emptyset \quad \text{if and only if}\quad [I_j]\in
    \Cl(\calO)^{[\gra]}.  
\end{equation}
If $X(\calO_j)\neq \emptyset$, then $\calO_j^\times$ acts simply transitively
on $X(\calO_j)$ 
from the right by multiplication.  Thus 
\begin{equation}
  \label{eq:18}
\frac{\abs{Y(\calO_j)}}{w(\calO_j)}=
  \begin{cases}
  1 \qquad &\text{if } [I_j]\in
    \Cl(\calO)^{[\gra]},\\
    0 \qquad &\text{otherwise.}
  \end{cases}
\end{equation}

We count the cardinality $\abs{Y(\calO_j)}$ in another way. 
By Corollary~\ref{endo.1}, each $\lambda \in X(\calO_j)$ gives rise to a CM
proper $A$-order $B_\lambda:=F(\lambda)\cap \calO_j$ such that
$[\gra]\in \ker(i_{B_\lambda/A})$. In other words, the $A$-isomorphism class $\dbr{B_\lambda}$
belongs to the set $\scrB_{[\gra]}(\calO_j)$ defined in
(\ref{eq:19}). Thus we have 
\begin{align}
  \label{eq:23}
   X(\calO_j)&=\coprod_{\dbr{B}\in
     \scrB_{[\gra]}(\calO_j)}X(\calO_j,\dbr{B}),\qquad \text{where}\\
X(\calO_j,\dbr{B})&:=\{\lambda \in X(\calO_j)\mid B_\lambda \simeq
                    B\}\subseteq  X(\calO_j)\subset D^\times. 
\end{align}
Clearly, each $X(\calO_j,\dbr{B})$ is $A^\times$-stable, so we set
$Y(\calO_j,\dbr{B}):=X(\calO_j,\dbr{B})/A^\times$.  
Let $K$ be the fractional field of $B$. Similar to (\ref{eq:20}), we
define 
\begin{equation}
  \label{eq:22}
    X(B):=\{\lambda \in K^\times\mid \gra B=\lambda
                B\},\qquad 
    Y(B):=X(B)/A^\times.
\end{equation}
Then $\abs{Y(B)}=w(B):=[B^\times:A^\times]$. 

For each $\dbr{B}\in
     \scrB_{[\gra]}(\calO_j)$, we have an $A^\times$-equivariant surjective map 
\begin{equation}
  \label{eq:24}
  X(B)\times \Emb(B,\calO_j)\to X(\calO_j,\dbr{B}), \qquad (\lambda,
  \varphi)\mapsto \varphi(\lambda).
\end{equation}
Let $\varphi$ and $\varphi'$ be two distinct optimal embeddings of $B$ into
$\calO_j$. If $\varphi(B)\neq
\varphi'(B)$, then $\varphi(X(B))\cap \varphi'(X(B))=\emptyset$. On
the other hand, $\varphi(B)=\varphi'(B)$ if and only if 
\[\varphi'=\bar\varphi, \qquad B=\bar B.\]
Therefore, if $\delta(B)=0$, then the map in (\ref{eq:24}) is
bijective; otherwise it is two-to-one. It follows that from
(\ref{eq:25}) that 
\begin{equation}
  \label{eq:26}
  \frac{\abs{Y(\calO_j,\dbr{B})}}{w(\calO_j)}=\frac{1}{2}(2-\delta(B))\frac{w(B)m(B,
  \calO_j)}{w(\calO_j)}=\frac{1}{2}(2-\delta(B))m(B, \calO_j, \calO_j^\times).
\end{equation}
A priori, (\ref{eq:26}) holds for all $\dbr{B}\in
\scrB_{[\gra]}(\calO_j)$, but we may extend it to all members of
$\scrB_{[\gra]}$. Indeed, if $\dbr{B'}\in \scrB_{[\gra]}-\scrB_{[\gra]}(\calO_j)$,
then both ends of (\ref{eq:26}) are zero since the sets $\Emb(B', \calO_j)$ and
$X(\calO_j,\dbr{B'})$ are both empty. Similarly, the 
disjoint union in (\ref{eq:23}) can be taken to range over all $
\dbr{B}\in\scrB_{[\gra]}$.

By \cite[Theorem 5.11, p.~92]{vigneras} (see also
\cite[Lemma~3.2]{wei-yu:classno} and
\cite[Lemma~3.2.1]{xue-yang-yu:ECNF}), 
\begin{equation}
  \label{eq:27}
  \sum_{j=1}^h m(B, \calO_j, \calO_j^\times) =h(B)\prod_{\ell}m_\ell(B).
\end{equation}
Summing (\ref{eq:26}) over all $\dbr{B}\in \scrB_{[\gra]}$ and  $1\leq j \leq h$, we obtain 
\begin{equation}
  \label{eq:28}
 \sum_{\dbr{B}\in \scrB_{[\gra]}} \sum_{j=1}^h  \frac{\abs{Y(\calO_j,\dbr{B})}}{w(\calO_j)}=\sum_{\dbr{B}\in \scrB_{[\gra]}}\frac{1}{2}(2-\delta(B))h(B)\prod_{\ell}m_\ell(B).
\end{equation}
We exchange the order of summation on the left of (\ref{eq:28}) and apply
(\ref{eq:18}) to get 
\begin{equation}
  \label{eq:29}
  \sum_{j=1}^h\sum_{\dbr{B}\in \scrB_{[\gra]}}
  \frac{\abs{Y(\calO_j,\dbr{B})}}{w(\calO_j)}=
  \sum_{j=1}^h\frac{\abs{Y(\calO_j)}}{w(\calO_j)}=  \abs{ \Cl(\calO)^{[\gra]}}. 
\end{equation}
This concludes the proof of the proposition.
\end{proof}
\begin{proof}[Proof of Theorem~\ref{thm:orbit-num-formula}]
The general formula for $r(\calO)$ follows by combining
Burnside's lemma with Proposition~\ref{prop:fixed-number}. Assume
that $A=O_F$. Then  $\delta(B)=1$ for all $B$. Recall that 
$\scrB_{[\gra]}\neq \emptyset$ if and only if $1\neq [\gra]\in \grC_2(F)$.
Thus \eqref{eq:49} is a direct consequence of (\ref{eq:30}). 
\end{proof}

As a consequence of Theorem~\ref{thm:orbit-num-formula}, we prove the following slight generalization of \cite[Corollaire~1.1, p.~81]{vigneras:ens}, which is stated for Eichler orders. 
\begin{cor}
Let $D$ be a totally definite quaternion algebra over a totally real field $F$, and $\calO$ be an $O_F$-order in $D$. Then $h(F)$ divides $2h(\calO)$. 
\end{cor} 
\begin{proof}
It is well known
\cite[Corollary~5.4]{Conner-Hurrelbrink} that $h(B)/h(F)$ is integral for
any CM $O_F$-order $B$. So the corollary follows directly from \eqref{eq:49}.
\end{proof}
\begin{cor}\label{cor:type-numer-tot-def-unram}
  Let $F$ be a totally real number field of even degree, and $D$
  be the unique totally definite 
  quaternion $F$-algebra up to isomorphism   
  unramified at all finite places of $F$.  We have 
  \begin{equation}
    \label{eq:33}
t(D)=\frac{1}{h(F)}\left(h(D)+ \frac{1}{2}\sum_{\dbr{B}\in \scrB} h(B) \right),    
  \end{equation}
\end{cor}
\begin{proof}
  As shown in the proof of Corollary~\ref{2.7}, we have
  $t(D)=r(\calO)$ for any maximal $O_F$-order $\calO$ in $D$.  By
  \cite[Theorem~II.3.2]{vigneras}, $m_\ell(B)=1$ for every prime
  $\ell\in \bbN$. The corollary then follows from \eqref{eq:49}. 
\end{proof}

Of course, (\ref{eq:33}) is just a special case of the general
type number formula for Eichler orders \cite[Theorem~A]{Pizer1973} (see also
\cite[Corollary~V.2.6]{vigneras}).  However, it describes  more concretely
the CM $O_F$-orders $B$ with nonzero contribution to
the type number. Thanks to this, we prove in Theorem~\ref{thm:integrality} that $h(D)/h(F)$ is
an integer, which is asserted by Vign\'eras \cite[Remarque, p.~82]{vigneras:ens}.  In light of  \eqref{eq:33}, it is enough to show the cardinality of
the following subset of $\scrB$ is even: 
\begin{equation}\label{eq:44}
\scrB_\odd:=\{\dbr{B}\in \scrB\mid h(B)/h(F)  \text{ is
  odd}\}.
\end{equation}
We will work out the details in Section~\ref{sec:integrality}.

\section{Calculation of type numbers}
\label{sec:03}

Throughout this section, we fix a prime number $p$. Let $q$ be an odd power
 of $p$, and $\Sp(\sqrt{q})$ and $\calT(\Sp(\sqrt{q}))$ be as in
Section~\ref{sec:01}. Let $F=\Q(\sqrt{p})$, $A=\Z[\sqrt{p}]$, and 
$D=\Bp\otimes_\Q F$ be the unique definite quaternion $F$-algebra
unramified at all finite places. We may regard $D$ as the
endomorphism algebra $\End_{\F_q}^0(X_0)$ of a member $X_0/\F_q$ in
$\Sp(\sqrt{q})$.


Let $\mathbb{O}_1$ be a maximal $O_F$-order in $D$. 
When $p\equiv 1 \pmod 4$, one has $A\neq O_F$, and $A/2O_F\cong
\ff_2$. In this case, we define $A$-orders $\bbO_r$ in $D$ for
$r\in\{8,16\}$ as follows. 
Let $\mathbb{O}_8, \mathbb{O}_{16}\subset \mathbb{O}_1$ be the proper
$A$-orders such that
\begin{gather}
  (\mathbb{O}_8)_2=
\begin{pmatrix}
  A_2 & 2 O_{F_2} \\  O_{F_2} & O_{F_2}\\
\end{pmatrix},\qquad (\mathbb{O}_{16})_2=\Mat_2(A_2),
\label{eq:bbOr1} \\   
(\mathbb{O}_r)_\ell=(\mathbb{O}_1)_\ell\qquad 
\text{for }  r\in \{8,16\}  \text{ and every prime } \ell\neq 2. \label{eq:bbOr2}
\end{gather}
Here for any $\Z$-module $M$ and any prime $\ell\in \bbN$, we write $M_\ell$
for $M\otimes \Z_\ell$. 
The order $\mathbb{O}_r$ is of 
index $r$ in $\mathbb{O}_1$.

By \cite[Theorem 6.1.2]{xue-yang-yu:ECNF} and \cite[Theorem
1.3]{xue-yang-yu:sp_as}, there is a bijection (depending on the choice of
a base point in each genus)
\begin{equation}
  \label{eq:3.3}
  \Sp(\sqrt{q})\simeq \coprod_{r} \Cl(\bbO_r),
\end{equation}
where $r\in \{1,8,16\}$ if $p\equiv 1 \pmod 4$ and $r=1$ otherwise. 
If $[I]$ is an ideal class corresponding to a 
class $[X]$ in $\Sp(\sqrt{q})$ in this bijection,  then  $\calO_l(I)\simeq \End(X)$.
Therefore, there is a bijection $\calT(\Sp(\sqrt{q}))\simeq \coprod_r
\calT(\bbO_r)$, and 
\begin{equation}
  \label{eq:3.4}
  T(p):=|\calT(\Sp(\sqrt{q}))|=
  \begin{cases}
    t(\bbO_1) & \text{for $p\not \equiv 1 \pmod 4$}; \\
    t(\bbO_1)+t(\bbO_8)+t(\bbO_{16}) & \text{for $p\equiv 1 \pmod 4$}. 
  \end{cases}
\end{equation}


Put $\bbO_4:=O_F \bbO_8$, which is of index $4$ in
$\bbO_1$. It is an Eichler order of level $2O_F$, and 
one has $\bbO_4=\bbO_1\cap 
w \bbO_1 w^{-1}$, where $w=
  \begin{pmatrix}
    0 & 2 \\ 1 & 0 \\
  \end{pmatrix}$. 
Since $D$ is unramified at all finite places, $\calN(\wh \bbO_1)
=\wh F^\times \wh \bbO_1^\times$. 
Clearly, $\calN(\wh \bbO_8)\subseteq \calN(\wh \bbO_4)$, and
$\calN(\wh\bbO_{16})\subseteq \calN(\wh\bbO_1)$. 
Direct calculations show that 
$\calN(\wh \bbO_8)=\wh F^\times \wh \bbO_4^\times$ 
and $\calN(\wh \bbO_{16})=\wh F^\times \wh \bbO_{16}^\times$ 
for $p\equiv 1 \pmod 4$. Therefore, we have natural bijections 
\begin{equation}
  \label{eq:3.5}
  \calT(\bbO_1)\simeq \frac{\Cl(\bbO_1)}{\Pic(O_F)}, \quad \calT(\bbO_8)
  \simeq \frac{\Cl(\bbO_4)}{\Pic(O_F)}, \quad \calT(\bbO_{16})\simeq
  \frac{\Cl(\bbO_{16})}{\Pic(A)}. 
\end{equation}

\begin{prop}\label{3.1}
  For every prime $p$, we have $t(\bbO_1)=h(\bbO_1)/h(F)$. Moreover,  if $p\equiv 1 \pmod 4$, then
 $t(\bbO_8)=h(\bbO_4)/h(F)$ and 
 $t(\bbO_{16})=h(\bbO_{16})/h(A)$.
\end{prop}
\begin{proof}
It is known that $h(F)$ is odd for all prime $p$ 
(see \cite[Corollary (18.4), p.~134]{Conner-Hurrelbrink}). By
Corollary~\ref{endo.4}, $\Pic(O_F)$ acts freely on $\Cl(\bbO_1)$ and
$\Cl(\bbO_4)$. For $r=16$, we have $\wt A=A$. On the other hand, one
easily computes 
that $h(A)=h(F)$ if $p\equiv 1 \pmod 8$, and $h(A)=3h(F)/\varpi$ if
$p\equiv 5 \pmod 8$, where $\varpi:=[O_F^\times:A^\times]$.  
Particularly, $h(A)$ is odd. Then the action of $\Pic(A)$ on
$\Cl(\bbO_{16})$ is also free by Corollary~\ref{endo.4} again.
The proposition then follows from (\ref{eq:3.5}).    
\end{proof}

\def\Mass{{\rm Mass}}
\def\Emb{{\rm Emb}}

For any square free integer $d\in \bbZ$, we write $h(d)$ for the class
number $h(\Q(\sqrt{d}))$. The discriminant of $F/\Q$ is denoted by $\grd_F$. We also set $K_j:=F(\sqrt{-j})=\Q(\sqrt{p},\sqrt{-j})$ for $j\in \{1,2,3\}$. 


\begin{lemma}\label{3.2}
  Assume that $p\equiv 1\pmod 4$. Then  
  \begin{equation}
  \label{eq:3.6}
  h(\bbO_4)=
  \begin{cases}
    \frac{9}{2} \zeta_F(-1)h(F)+\frac{1}{4} h(K_1) & \text{for
    $p\equiv 1\pmod 8$;} \\ 
    \frac{5}{2} \zeta_F(-1)h(F)+\frac{1}{4} h(K_1)+\frac{2}{3} h(K_3)
    & \text{for $p\equiv 5 
    \pmod 8$.}  
  \end{cases}
\end{equation}
In particular, $h(\bbO_4)=1$ if $p=5$.   
\end{lemma}

The special value $\zeta_F(-1)$ of the Dedekind zeta-function
$\zeta_F(s)$ can be calculated by Siegel's formula \cite[Table~2, p.~70]{Zagier-1976-zeta}. It is also known that $\zeta_F(-1)=\frac{1}{24}B_{2,\chi}$, where $B_{2,\chi}$ is the generalized Bernoulli number attached to the primitive quadratic Dirichlet character (of conductor $\grd_F$) associated to the quadratic field $F$.  This can be obtained, for example, by combining \cite[formula (16), p.~65]{Zagier-1976-zeta} together with \cite[Exercise~4.2(a)]{Washington-cyclotomic}.

\begin{proof}
  In this case $\bbO_4$ is an Eichler order of square-free level
  $2O_F$. Let $m_2(B)$ be as in (\ref{eq:35}), with $\ell=2$ and
  $\calO=\bbO_4$.  Using the formula \cite[p.~94]{vigneras}(cf. \cite[Section~3.4]{xue-yang-yu:ECNF}), 
  we compute that $m_2(O_{K_1})=1$ and 
  $m_2(O_{K_3})=1-\Lsymb{2}{p}$, where $\Lsymb{2}{p}$ is the Legendre
  symbol. Let $\zeta_n$ be a primitive $n$-th root of unity for each
  $n>1$. When $p=5$, we also need to compute
  $m_2(\Z[\zeta_{10}])$, which is zero. Using Eichler's class number
  formula (see \cite[Chapter V]{vigneras}) and the data
  \cite[Sect.~2.8 and Prop.~3.1]{xue-yang-yu:num_inv}, 
  we obtain a formula for $h(\bbO_4)$. One can also apply Vign\'eras's  formula
  \cite[Theorem 3.1]{vigneras:ens} to the Eichler order
  $\bbO_4$. 
\end{proof}

\begin{proof}[Proof of Theorem~\ref{1.2}]
By \cite[Theorem 3.1]{vigneras:ens} 
(cf.~\cite[Proposition~6.2.10]{xue-yang-yu:ECNF}), we have 
$h(\bbO_1)=1, 2, 1$ for $p=2,3, 5$, respectively.
Since $h(F)=h(p)=1$ for $p=2,3$ and $5$, by Proposition~\ref{3.1} we obtain 
\begin{equation}
  \label{eq:tO1}
  t(\bbO_1)= 1,2,1 \quad \text{
  for $p=2,3,5$, respectively.}
\end{equation}
By Lemma~\ref{3.2} and \cite[Proposition~6.2.10]{xue-yang-yu:ECNF}, 
we have $h(\bbO_4)=h(\bbO_{16})=1$ for $p=5$. Therefore,
\begin{equation}
  \label{eq:O8-16}
 t(\bbO_8)=t(\bbO_{16})=1 \quad \text{for $p=5$}. 
\end{equation}

Assume that $p\equiv 3 \pmod 4$ and $p\ge 7$. By \cite[(6.17)]{xue-yang-yu:ECNF} and Proposition~\ref{3.1}, we have 
\begin{equation}
  \label{eq:type-O1-p3mod4}
  t(\bbO_1)= \frac{\zeta_F(-1)}{2} +
     \left(13-5\left(\frac{2}{p}\right)
    \right)\frac{h(-p)}{8}+\frac{h(-2p)}{4}+\frac{h(-3p)}{6}. 
\end{equation}

Assume $p\equiv 1 \pmod 4$ and $p\ge 7$. By \cite[(6.16)]{xue-yang-yu:ECNF} and Proposition~\ref{3.1}, we obtain
  \begin{equation}
    \label{eq:type-O1-p1mod4}
    t(\bbO_1)=\frac{\zeta_F(-1)}{2} +
     \frac{h(-p)}{8}+\frac{h(-3p)}{6}.      
  \end{equation} 
By Proposition~\ref{3.1} and Lemma~\ref{3.2}, 
we obtain
\begin{equation}
  \label{eq:tO8}
  t(\bbO_8)=
  \begin{cases}
    \frac{9}{2}\zeta_F(-1)+h(-p)/8 & \text{for $p\equiv 1 \pmod 8$;} \\
    \frac{5}{2}\zeta_F(-1)+h(-p)/8+h(-3p)/3  & \text{for $p\equiv 5
    \pmod 8$.} 
  \end{cases}
\end{equation}
Here we use the result of Herglotz
  \cite{MR1544516} (see also \cite[Chapter~3]{vigneras:ens} and
  \cite[Subsection~2.10]{xue-yang-yu:num_inv}) to factor out $h(F)$ from $h(K_1)$ and $h(K_3)$.  
By the formulas for 
$h(\bbO_{16})$ 
\cite[Section~6.2.4]{xue-yang-yu:ECNF} and Proposition~\ref{3.1},
we get 
\begin{equation}
  \label{eq:tO16}
  t(\bbO_{16})=  
    \begin{cases}
    3 \zeta_F(-1)+h(-p)/4+h(-3p)/2 & \text{for $p\equiv 1 \pmod 8$;} \\
    5 \zeta_F(-1)+h(-p)/4+h(-3p)/6  & \text{for $p\equiv 5 \pmod 8$.}
  \end{cases}
\end{equation}

Theorem~\ref{1.2} follows from (\ref{eq:3.4}) and (\ref{eq:tO1}) --
(\ref{eq:tO16}).
\end{proof}

\begin{remark}\label{3.3}
  (1) Comparing our formulas for $h(\bbO_8)$ and $h(\bbO_4)$, we see
  that $h(\bbO_4)=h(\bbO_8)$ if $p\equiv 1\pmod 8$, or if $p\equiv
  5\pmod 8$ and $\varpi=[O_F^\times:A^\times]=3$. In both
  cases,  we have $h(A)=h(F)$, and hence
  \begin{equation}
    \label{eq:3.13}
    \begin{split}
    T(p)&=\frac{h(\bbO_1)}{h(F)}+\frac{h(\bbO_4)}{h(F)}+
   \frac{h(\bbO_{16})} {h(A)} \\
    &=\frac{h(\bbO_1)}{h(F)}+\frac{h(\bbO_8)}{h(F)}+
   \frac{h(\bbO_{16})}{h(F)}=\frac{H(p)}{h(F)}. \\
    \end{split}    
  \end{equation}


  (2) Note that $\bbO_4$ is also defined for $p\not \equiv 1\pmod 4$, 
   as an Eichler order of (non-square-free) level $2O_F$. We 
   calculate $h(\bbO_4)$ and obtain that $h(\bbO_4)=1,2$ for $p=2,3$,
   and for $p\ge 7$ and $p\equiv 3 \pmod 4$, 
   \begin{equation}
     \label{eq:3.14}
     h(\bbO_4)=3 \zeta_F(-1) h(F)+\left (15-
     3\left(\frac{2}{p}\right)\right)\frac{h(K_1)}{4}.  
   \end{equation}
   More work is needed for computing the numbers of local optimal
   embeddings at $2$. We omit the details as
   (\ref{eq:3.14}) is not used in the present paper.  
\end{remark}


Assume that $p\equiv 1\pmod{4}$. Let us calculate $r(\bbO_8)$, the number of orbits for the
$\Pic(A)$-action on $\Cl(\bbO_8)$. As mentioned above, if either
$p\equiv 5\pmod{8}$ and $\varpi=3$ or $p\equiv 1\pmod{8}$, then we
have a natural isomorphism $\Pic(A)\simeq \Pic(O_F)$. Since the class
number of $F$ is odd, $\Pic(A)$ acts freely on $\Cl(\bbO_8)$ in this
case by
Lemma~\ref{endo.2}. Combining with part (1) of Remark~\ref{3.3}, we obtain
\begin{equation}
  \label{eq:38}
  r(\bbO_8)=\frac{h(\bbO_4)}{h(F)}=t(\bbO_8).  
\end{equation}
Surprisingly, this holds for the case that $p\equiv 5\pmod{8}$ and
$\varpi=1$ as well (to be proved in Subsection~\ref{sect:orbit-O8}). 
But before that, we relax the condition on $p$
a little bit and study a more general real quadratic field
$F=\Q(\sqrt{d})$.

\begin{ex}\label{ex:ker-cyclic-order-3}
  Let $d\in \bbN$ be a square free positive integer congruent to $5$
  modulo $8$. Whether the fundamental unit $\varepsilon$ of
  $F=\Q(\sqrt{d})$ lies in $A=\Z[\sqrt{d}]$ or not is a classical
  problem that dates back to Eisenstein \cite{Eisenstein1844}, and there 
  seems to be no simple criterion on $d$ for it. However, it is
  known \cite{Alperin, Stevenhagen} that the number of $d$ for each
  case is infinite. By
  \cite[Section~3]{Alperin}, the kernel of $i_{O_F/A}:\Pic(A)\to
  \Pic(O_F)$ is generated by the locally principal $A$-ideal class $[\gra]$ represented by $\gra=4\Z +
  (1+\sqrt{d})\Z$. Moreover, $\gra^3=8A$, and $\gra$ is
  non-principal if and only if $\varepsilon \in A$. 


  Let $K=\Q(\sqrt{d}, \sqrt{-3})$, and
  $B_{3,2}=A[\sqrt{-3},\frac{1+\sqrt{d}}{2}\zeta_6]$ with
  $\zeta_6=(1+\sqrt{-3})/2$. One calculates that
  \begin{equation}
    \label{eq:42}
B_{3,2}=\Z+\Z\sqrt{d}+\Z\frac{\sqrt{d}+\sqrt{-3}}{2}+\Z\frac{(1+\sqrt{d})(1+\sqrt{-3})}{4}\subset
  O_K. 
  \end{equation}
In particular, $B_{3,2}$ is a CM proper $A$-order with
$[O_K:B_{3,2}]=2$ and $\delta(B_{3,2})=0$. 
  We have $2O_K\subset B$, and
  $B/2O_K=\F_4\oplus \F_2\subset \F_4\oplus \F_4=O_K/2O_K$
  (cf.~\cite[Section~4.8]{xue-yang-yu:num_inv}).   Note that
  $\zeta_6\not\in B_{3,2}$, otherwise $B_{3,2}=O_K$ by
  \cite[Exercise~II.42(d), p.~51]{MR0457396}. Thus, we have 
\[3=[(O_K/2O_K)^\times: (B_{3,2}/2O_K)^\times]\geq [O_K^\times:
B_{3,2}^\times]\geq 3,\]
and hence $[O_K^\times:
B_{3,2}^\times]=3$.  It follows from \cite[Theorem~I.12.12]{Neukirch-ANT} that $h(B_{3,2})=h(O_K)$, and the canonical map
$\Pic(B_{3,2})\to \Pic(O_K)$ is an isomorphism. Therefore,
$\ker(i_{B_{3,2}/A})=\ker(i_{O_K/A})\supseteq \ker(i_{O_F/A})=\dangle{[\gra]}$.

Now assume that $3\nmid d$, and  $\varepsilon\in A$ so that $\gra$ is
non-principal. Then $K/F$ is ramified at every place
of $F$ above
  $3$. Hence  $O_K^\times =O_F^\times
  \bmu(K)$ and 
  $i_{K/F}: \Pic(O_F)\to \Pic(O_K)$ is an embedding by 
  \cite[Lemma~13.5]{Conner-Hurrelbrink}. Therefore, we have
\[\ker(i_{B_{3,2}/A})=\dangle{[\gra]}\simeq \zmod{3}.  \]  
It is known \cite[Theorem~4.1]{Alperin} that there are infinitely many
square free $d$ of the form
  $4n^2+1$ with an odd $n>3$ such that $\varepsilon\in A$. In particular, there are
  infinitely many $d$ satisfying our assumptions.
\end{ex}

\begin{sect}\label{sect:orbit-O8}
We return to the case that $F=\Q(\sqrt{p})$ with $p\equiv 5\pmod{8}$. 
Assume that $\varepsilon\in A=\Z[\sqrt{p}]$ so that
$\ker(i_{O_F/A})=\dangle{[\gra]}\simeq \zmod{3}$, with
$\gra=4\Z+(1+\sqrt{p})\Z$ as in
Example~\ref{ex:ker-cyclic-order-3}. In order to compute $r(\bbO_8)$, we list
all CM proper $A$-orders $B$ with $\ker(i_{B/A})$ nontrivial. Let $K$
be the fractional field of $B$. Since $h(F)$ is odd,  
the morphism  $i_{K/F}:
\Pic(O_F)\to \Pic(O_K)$ is injective, and hence 
$\ker(i_{B/A})\subseteq \ker(i_{O_F/A})$.  By the proof of
Proposition~\ref{prop:finiteness-CM-proper-A-order},  necessarily
$[O_K^\times:O_F^\times]>1$, and there exists a unit $u\in O_K^\times$
such that $u\not\in O_F^\times$ and $u\gra/2\subseteq B$. According to
\cite[Subsection~2.8]{xue-yang-yu:num_inv}, $K$ coincides with either
$F(\sqrt{-1})$ or $F(\sqrt{-3})$, and $O_K^\times=O_F^\times\bmu(K)$
in both cases. By the assumption, $O_F^\times=A^\times$. 

First,
suppose that $K=F(\sqrt{-1})$. Then $O_K^\times/A^\times$ is a cyclic
group of order $2$ generated by the coset of $\sqrt{-1}$. Note that 
$A[\gra\sqrt{-1}/2]\supseteq O_F$, because
\[\left(\frac{(1+\sqrt{p})\sqrt{-1}}{2}\right)^2=-\left(\frac{p-1}{4}+\frac{1+\sqrt{p}}{2}\right).\]
Hence there is no CM \emph{proper} $A$-order $B$ in $F(\sqrt{-1})$ with
$\ker(i_{B/A})$ nontrivial. 

Second, suppose that $K=F(\sqrt{-3})$. Then $O_K^\times/A^\times$ is a cyclic
group of order $3$ generated by the coset of $\zeta_6$. We have $A[\gra\zeta_6/2]=B_{3,2}$ in
(\ref{eq:42}), and $A[\gra\zeta_6^{-1}/2]=\bar B_{3,2}$, the complex
conjugate of $B_{3,2}$.

Therefore, up to isomorphism, $B=B_{3,2}$ is the
unique CM proper $A$-order with $\ker(i_{B/A})$ nontrivial. The same
proof as that in \cite[Subsection~6.2.5]{xue-yang-yu:ECNF} shows that
$m_2(B_{3,2})=1$ for $\calO=\bbO_8$. It follows
from Theorem~\ref{thm:orbit-num-formula} that 
\begin{equation}
  \label{eq:43}
  r(\bbO_8)=\frac{1}{3h(F)}(h(\bbO_8)+2h(B_{3,2}))=\frac{5}{2}\zeta_F(-1)+\frac{1}{8}h(-p)+\frac{1}{3}h(-3p). 
\end{equation}
Comparing with (\ref{eq:tO8}), we find that (\ref{eq:38}) holds in the
current setting as well. 
\end{sect}






\section{Integrality of $h(D)/h(F)$ when $\grd(D)=O_F$}
\label{sec:integrality}

Let $F$ be a totally real number field, and $D$ be a totally definite
quaternion algebra over $F$. The goal of this section is two-fold:
first, we characterize all the CM $O_F$-orders $B$ for which the
kernel of $i_{B/O_F}:\Pic(O_F)\to \Pic(B)$ is nontrivial; second, we
show that $h(D)/h(F)$ is integral when $[F:\Q]$ is even and $D$ is unramified at all finite
places of $F$ (i.e. the reduced discriminant $\grd(D)=O_F$). 

For simplicity, we identify $F$ with a subfield of $\C$, and consider
CM extensions of $F$ as subfields of $\C$ as well. This way two CM $O_F$-orders
are isomorphic if and only if they are the same. The set $\scrB$ defined in (\ref{eq:41}) with $A=O_F$ becomes to  
\[\scrB=\{B\mid B \text{ is a CM $O_F$-order, and }
\ker(i_{B/O_F})\neq \{1\}\}.
\]
By Lemma~\ref{lem:kernel-2-torsion}, $\ker(i_{B/O_F})\simeq \zmod{2}$
for every $B\in \scrB$.  In particular, if $h(F)$ is
odd, then $\scrB=\emptyset$ and $h(D)/h(F)$ is an integer by Corollary~\ref{2.7}. So we shall focus on the
case where $h(F)$ is even. 
The \emph{finitely many} CM extensions $K/F$ with
$\ker(i_{K/F})$ nontrivial are classified in
\cite[Section~14]{Conner-Hurrelbrink}. It remains to characterize
all $O_F$-orders $B\subseteq O_K$ with $\ker(i_{B/O_F})=\ker(i_{K/F})$ for
each such $K$.




\begin{lem}\label{lem:suborder-nontrivial-kernel}
  Let $K/F$ be a CM-extension with $\ker(i_{K/F})\neq \{1\}$, and
  let $[\gra]\in \ker(i_{K/F})$ be the unique nontrivial ideal class so
  that $\gra O_K=\lambda O_K$ for some $\lambda\in K^\times$.
  \begin{enumerate}[(i)]
  \item Suppose that $K\neq F(\sqrt{-1})$. Then $[\gra]\in
    \ker(i_{B/O_F})$ if and only if $B$ contains the
    $O_F$-order $B_0:=O_F\oplus\grI$, where $\grI$ denotes the purely imaginary
  $O_F$-submodule $\{z\in O_K\mid
  \bar{z}=-z\}$ of $O_K$. 
\item Suppose that $K=F(\sqrt{-1})$. Let $n=\max\{ m \in \bbN\mid
  \Q(\zeta_{2^m})\subseteq K\}$, where $\zeta_{2^m}$ denotes a primitive
  $2^m$-th root of unity.  Put
  $\eta=\zeta_{2^n}$, and $\grI_\eta=\{z\in O_K\mid \bar{z}=\eta z\}$.
 Then $[\gra]\in
    \ker(i_{B/O_F})$ if and only if $B$ contains the
    $O_F$-order $B_0:=O_F\oplus\grI_\eta$. Moreover,
    $\dangle{\eta}\subseteq B_0^\times$, so $B_0$ does not depend on
    the choice of $\zeta_{2^n}$. 
  \end{enumerate}
\end{lem}

\begin{proof}
  By Lemma~\ref{lem:char-CM-order-nontrivial-i},
  $[\gra]\in \ker(i_{B/A})$ if and only if there exits
  $u\in O_K^\times$ such that 
  $u\gra/\lambda\subset B$.  Let $\bmu(K)$
  be the group of roots of unity in $K$. Because $\ker(i_{K/F})$ is
  nontrivial, it follows from \cite[Section~13, p.~68]{Conner-Hurrelbrink}
  that the CM-extension $K/F$ is of type I,
  i.e. $O_K^\times=O_F^\times\bmu(K)$. Write $u=v\xi$ with
  $v\in O_F^\times$ and $\xi\in \bmu(K)$. Then
  $\xi\gra/\lambda=u\gra/\lambda\subset B$. Therefore,   $[\gra]\in \ker(i_{B/O_F})$ if and only if there exists a
  root of unity $\xi\in \bmu(K)$ such that
  $\xi\gra/\lambda\subset B$.


(i)  Now suppose  that $K\neq F(\sqrt{-1})$.   By
\cite[Section~14]{Conner-Hurrelbrink}, replacing $\lambda$ by $\lambda u'$ for some $u'\in O_K^\times$, we may assume 
  that $\bar{\lambda}=-\lambda$ and $\gra^2=-\lambda^2 O_F$. Clearly,
  $\gra/\lambda \subseteq\grI$, and
  $B_0:=O_F\oplus\grI$ is an $O_F$-suborder of
  $O_K$. By Lemma~\ref{lem:char-CM-order-nontrivial-i}, $[\gra]\in \ker(i_{B/O_F})$ for any $O_F$-order $B\supseteq B_0$.  
  
Now we prove the other direction. We first claim that $\gra/\lambda =\grI$.  
Since both sides are invertible
 $O_F$-modules, there exists a nonzero $O_F$-ideal $\grc\subseteq O_F$ such
 that $\gra/\lambda =\grc\grI$. Then
 $O_K=\frac{\gra}{\lambda}O_K=\grc\grI O_K\subseteq \grc O_K$. It
 follows that $\grc O_K=O_K$, and hence $\grc=O_F$ by
 Lemma~\ref{lem:eq-unit-ideal}. 

 Let $B$ be an $O_F$-suborder of $O_K$ such that
 $\xi\gra/\lambda\subseteq B$ for some $\xi\in \bmu(K)$. Replacing
 $\xi$ by $-\xi$ if necessary, we may assume that $\xi$ has odd
 order.  Then 
$\xi^2\in \xi^2O_F=(\xi\gra/\lambda)^2\subseteq B$, and hence
$\dangle{\xi}\subset B$. It follows that $\grI=\gra/\lambda\subseteq
B$, and hence $B\supseteq B_0$. 

(ii) Suppose that $K=F(\sqrt{-1})$. Note that both $\eta$ and
$\bar\eta$ belong to $B_0$ as 
$1+\bar\eta\in \grI_\eta$ and $\eta+\bar\eta\in O_F$. 
If both $x$ and $y$ are elements
of $\grI_\eta$, then  $xy\in O_F\bar\eta$ as $\ol{xy}=(xy \eta)\cdot \eta$ and $\overline{xy\eta}= xy\eta$. 
Thus $B_0$ is an $O_F$-suborder of
$O_K$, so $\dangle{\eta}\subset B_0^\times$.  By
\cite[Corollary~14.7]{Conner-Hurrelbrink} and the Remark on
\cite[p.~88]{Conner-Hurrelbrink}, the unique nontrivial ideal class in
$\ker(\Pic(O_F)\to \Pic(O_K))$ is represented by an $O_F$-ideal $\gra$ such
that $\gra^2=(2+\eta+\bar\eta)O_F$ and $\gra O_K=(1+\eta)O_K$. 
We have $\gra/(1+\eta)\subseteq \grI_\eta\subset B_0$.
By Lemma~\ref{lem:char-CM-order-nontrivial-i} again, $[\gra]\in \ker(i_{B/O_F})$ for any order $B\supseteq B_0$.  


Conversely, let $B\subseteq O_K$ be an $O_F$-suborder such that
 $\xi\gra/(1+\eta)\subseteq B$ for some $\xi\in \bmu(K)$. 
 The same proof as that in
(i) shows that $\gra/(1+\eta)=\grI_\eta$. Note that $2+\eta+\bar \eta\in \gra^2$.
 We have
 \begin{equation}
   \label{eq:2}
\xi^2\bar\eta=\frac{\xi^2(2+\eta+\bar\eta)}{(1+\eta)^2}\in
\left(\frac{\xi\gra}{1+\eta}\right)^2\subseteq B.
 \end{equation}
Write
$\xi=\xi'\eta^r$ such that $s:=\ord(\xi')$ is odd. Then 
\[\dangle{\eta}=\dangle{\eta^{s(2r-1)}}=\dangle{(\xi^2\bar{\eta})^s}\subseteq
\dangle{\xi^2\bar\eta}\subseteq B^\times.\]
It follows from (\ref{eq:2}) that $\xi'^2\in B$, and hence $\xi'\in B$ as
well. We conclude that $\xi\in B$, and 
$\grI_\eta=\gra/(1+\eta)\subseteq B$. This completes the proof of the lemma.
\end{proof}

\begin{cor}\label{cor:conductor}
  Keep the notation and assumptions of
  Lemma~\ref{lem:suborder-nontrivial-kernel}. The
  index $[O_K^\times: B^\times]$ is odd for any $O_F$-order
  $B\subseteq O_K$ with $[\gra]\in \ker(i_{B/O_F})$. If $K\neq
  F(\sqrt{-1})$, then $2O_K\subseteq B$; if $K=F(\sqrt{-1})$, then
  $(2-\eta-\bar\eta)O_K\subseteq B$. 
\end{cor}
\begin{proof}
  It is enough to prove the corollary for the case $B=B_0$. As pointed
  in the proof of Lemma~\ref{lem:suborder-nontrivial-kernel},
we have $O_K^\times=O_F^\times\bmu(K)$.  The 2-primary
subgroup of $\bmu(K)$ is contained in $B_0^\times$, so $[O_K^\times:
B_0^\times]$ is odd. Clearly, $2O_K\subseteq B_0$ if $K\neq
F(\sqrt{-1})$. Suppose that $K=F(\sqrt{-1})$. For any $\alpha\in O_K$,
write $\alpha=a+b$ with $a\in F$ and $b\in \grI_\eta\otimes_{O_F}
F\subseteq K$. We then have 
\begin{equation}
  \label{eq:1}
a=\frac{-\eta \alpha+\bar \alpha}{1-\eta},\qquad
b=\frac{\alpha-\bar\alpha}{1-\eta}. 
\end{equation}
Both $(2-\eta-\bar\eta)a$ and $(2-\eta-\bar\eta)b$ are
integral, and the corollary follows. 
\end{proof}


Now assume that $h(F)$ is even and consider the following set in (\ref{eq:44}):  
\begin{equation}\label{eq:45}
\scrB_\odd=\{B\in \scrB\mid h(B)/h(F) \text{ is odd}\}\subseteq \scrB.  
\end{equation}
Our goal is to show that $\abs{\scrB_\odd}$ is even. Clearly, if $B\in
\scrB_\odd$, then $O_K\in \scrB_\odd$ as well with $K$ being the fractional field of $B$. 
By a theorem of Kummer (cf.~\cite[Theorem 13.14]{Conner-Hurrelbrink}), 
if $h(F)$ is even and $h(K)/h(F)$ is odd, then $\ker (i_{K/F})$ is nontrivial. 
The CM extensions $K/F$ with odd relative class number
$h(K)/h(F)$ are classified in
\cite[Section~16]{Conner-Hurrelbrink}. We fix such a $K$ and
characterize all the $O_F$-suborders $B\subseteq O_K$ that lie in
$\scrB_\odd$.

\begin{prop}\label{prop:minimal-odd-rlCN-order}
  Let $K/F$ be a CM-extension with  $h(F)$ even and the relative class number
  $h(K)/h(F)$ \emph{odd}.  Denote by $\grf_\dia$ the product of all
  dyadic prime $O_F$-ideals that are unramified in $K$. 
An $O_F$-order $B\subseteq O_K$ has odd
  relative class number $h(B)/h(F)$ if and only if $B$ contains the
  $O_F$-order $B_\dia:=O_F+\grf_\dia O_K$. 
Moreover, $B_\dia=O_K$ if and
  only if both of the following conditions hold: 
  \begin{itemize}
  \item $F$ has a single dyadic prime, and 
  \item it is ramified in $K/F$. 
  \end{itemize}
\end{prop}
\begin{proof}
Since $h(F)$ is assumed to be even and $h(K)/h(F)$ is odd, we know by
\cite[Section~16]{Conner-Hurrelbrink} that
the $2$-primary subgroup $\Pic(O_F)[2^\infty]$ of $\Pic(O_F)$ is cyclic, and 
\begin{equation}
  \label{eq:3}
\ker(i_{K/F}:\Pic(O_F)\to
  \Pic(O_K))=\Pic(O_F)[2]\simeq \zmod{2}. 
\end{equation}
Let $[\gra]\in \ker (i_{K/F})$ be the unique nontrivial ideal class as
in Lemma~\ref{lem:suborder-nontrivial-kernel}. Then $i_{B/O_F}([\gra])\in \ker(\Pic(B)\twoheadrightarrow
\Pic(O_K))$ for any $O_F$-order $B\subseteq O_K$. If
$h(B)/h(F)$ is odd, then $h(B)/h(K)$ is odd and  $[\gra]\in
\ker (i_{B/O_F})$, and hence $B\supseteq B_0$ by
Lemma~\ref{lem:suborder-nontrivial-kernel}. By Corollary~\ref{cor:conductor}, 
for such an order  $B$, its
conductor $\grf(B)\subseteq O_F$ is a product of
dyadic primes of $O_F$, and $[O_K^\times:B^\times]$ is odd.  


The class number of an $O_F$-order $B\subseteq O_K$ is given by \cite[p.~75]{vigneras:ens}
\begin{equation}
  \label{eq:4}
  h(B)=\frac{h(K)\Nm_{F/\Q}(\grf(B))}{[O_K^\times:
    B^\times]}\prod_{\grp\mid \grf(B)}\left(1- \frac{\Lsymb{K}{\grp}}{\Nm_{F/\Q}(\grp)}\right),
\end{equation}
where $\Lsymb{K}{\grp}$ is the Artin's symbol \cite[p.~94]{vigneras}. More explicitly, 
\[\Lsymb{K}{\grp}=\begin{cases}
1 \quad &\text{if } \grp \text{ splits in } K;\\
0 \quad &\text{if } \grp \text{ ramifies in } K;\\
-1 \quad &\text{if } \grp \text{ is inert in } K.\\
\end{cases}\]
It follows that $h(B)/h(F)$ is odd if and only if 
\begin{enumerate}
\item $\grf(B)$ divides $\grf_0:=\grf(B_0)$ and is square-free; and 
\item every prime divisor $\grp$ of $\grf(B)$
  is unramified in $K$. 
\end{enumerate}
We claim that $\grf_0$ is divisible by every dyadic prime $\grp$ of
$F$ that is unramified in $K$. It is enough to prove it locally, so we
use a subscript $\grp$ to indicate completion at $\grp$. For example,
$F_\grp$ denotes the $\grp$-adic completion of $F$.


First, suppose that $K=F(\sqrt{-1})$.  Let $\grp$ be a dyadic prime of
$F$, and $\nu_\grp$ be its associated valuation.  By
\cite[63:3]{o-meara-quad-forms}, $K/F$ is unramified at $\grp$ if and
only if there exists a unit $u\in O_{F_\grp}^\times$ such that
$-1\equiv u^2 \pmod{4O_{F_\grp}}$. Assume that this is the case. Then
$O_{K_\grp}=O_{F_\grp}+O_{F_\grp}(u+\sqrt{-1})/2$. Given $m\in \bbN$,
we have $\grp^mO_{K_\grp}\subset (B_0)_\grp$ if and only if
$\grp^m(u+\sqrt{-1})/2\subset  (B_0)_\grp$. Write $(u+\sqrt{-1})/2=a+b$ with
$a\in F_\grp$ and $b\in \grI_\eta\otimes_{O_F}F_\grp$. Then by
(\ref{eq:1}), 
\[a=\frac{u}{2}-\frac{(1+\eta)\sqrt{-1}}{2(1-\eta)}, \qquad b=\frac{\sqrt{-1}}{1-\eta}.\]
We obtain that 
\begin{equation}
  \label{eq:5}
\nu_\grp(\grf_0)=\frac{1}{2}\nu_\grp(\Nm_{K/F}(1-\eta))=\frac{1}{2}\nu_\grp(2-\eta-\bar\eta)>0. 
\end{equation}

Next, suppose that $K\neq F(\sqrt{-1})$. By
\cite[p.~82]{Conner-Hurrelbrink}, $K=F(\sqrt{-\varsigma})$ for
some totally positive element $\varsigma\in F$ with
$\nu_\grq(\varsigma)\equiv 0\pmod{2}$ for every finite prime $\grq$ of
$F$. Thus for every $\grq$, there exists a unit $u\in
O_{F_\grq}^\times$ such that $K_\grq=F(\sqrt{-u})$.  Let $\grp$ be a
dyadic prime of $F$ that is unramified in $K$. Another application of
\cite[63:3]{o-meara-quad-forms} as above shows that
\begin{equation}
  \label{eq:6}
\nu_\grp(\grf_0)=\nu_\grp(2)>0. 
\end{equation}
This conclude the verification of the claim and proves the first part
of the proposition. 

For the last part of the proposition, recall that by
\cite[Theorem~16.1]{Conner-Hurrelbrink}, one of the following holds
for $K/F$:
\begin{itemize}
\item exactly one finite prime of $F$ ramifies in $K/F$ and it is
  dyadic; or
\item no finite prime of $F$ ramifies in $K/F$. 
\end{itemize}
If $\grf_\dia=O_F$, or equivalently, every dyadic prime
of $F$ ramifies in $K$,   then $F$ has a single dyadic 
prime and it ramifies in $K/F$.  The converse is obvious. 
\end{proof}

As mentioned before, the following theorem is asserted by Vign\'eras
\cite[Remarque, p.~82]{vigneras:ens}. 

\begin{thm}\label{thm:integrality}
  Let $F$ be a totally real number field of even degree over $\Q$, and
  $D$ be the totally definite quaternion $F$-algebra unramified at
  all the finite places of $F$. Then $h(D)/h(F)$ is integral. 
\end{thm}
\begin{proof}
The theorem follows from Corollary~\ref{endo.4} if  $h(F)$ is odd. Suppose that
  $h(F)$ is even.  We show that the cardinality of the set $\scrB_\odd$ in
  (\ref{eq:45}) is
  even. 
If $\Pic(O_F)[2]\not\simeq \zmod{2}$, then $\scrB_\odd=\emptyset$ by
the proof of Proposition~\ref{prop:minimal-odd-rlCN-order}. Suppose
further that $\Pic(O_F)[2]\simeq \zmod{2}$, and $K/F$ is a
CM-extension with odd relative class number $h(K)/h(F)$. Let
$\omega(\grf_\dia)$ be the number of prime factors of
$\grf_\dia$. Then
\begin{equation}
  \label{eq:7}
\abs{\{B\mid B_\dia\subseteq B\subseteq
O_K\}}=2^{\omega(\grf_\dia)}.  
\end{equation}
If $B_\dia\neq O_K$, then $\omega(\grf_\dia)>1$ and the right hand
side of (\ref{eq:7}) is even.  If $B_\dia=O_K$, then $F$ has a
single dyadic prime $\grp$, and it ramifies in $K$. By
\cite[Theorem~16.1]{Conner-Hurrelbrink} and 
\cite[Corollary~16.2]{Conner-Hurrelbrink}, $F$ admits exactly two
CM-extensions with odd relative class numbers,  and $\grp$ ramifies in
both of them. It follows that $\abs{\scrB_\odd}=2$ in this
case. Therefore, $\abs{\scrB_\odd}$ is even in all cases, as claimed. 
Now the theorem follows from
Corollary~\ref{cor:type-numer-tot-def-unram}. 
\end{proof}

\section*{Acknowledgments}
The second named author is grateful to Paul Ponomarev for answering
his questions on type numbers. 
The manuscript was prepared during the authors' visit
at M\"unster University. They thank Urs Hartl and the institution for
warm hospitality and excellent research environment.  
J.~Xue is partially supported by the 1000-plan program for young talents and
Natural Science Foundation grant \#11601395 of PRC. 
Yu is partially supported by the MoST grants 
104-2115-M-001-001MY3 and 107-2115-M-001-001-MY2.

\bibliographystyle{hplain}
\bibliography{TeXBiB}
\end{document}